\documentclass[11pt, a4paper]{amsart}
\usepackage{amsfonts,amssymb,amsmath,amsthm,amscd,mathtools,multicol,tikz, tikz-cd,caption,enumerate,mathrsfs,thmtools,cite,enumitem}

\usepackage{inputenc}
\usepackage[foot]{amsaddr}
\usepackage[pagebackref=true, colorlinks, linkcolor=blue, citecolor=red]{hyperref}
\usepackage{latexsym}
\usepackage{fullpage}
\usepackage{microtype}
\usepackage{subfiles} 
\setlist{topsep=0.5em, itemsep=1em} 

\usepackage{palatino}
\makeatletter
\makeindex 

\newcommand{\be}{\begin{equation}}
\newcommand{\ee}{\end{equation}}
\newcommand{\beano}{\begin{eqn*}} 
	\newcommand{\eeano}{\end{eqnarray*}}
\newcommand{\ba}{\begin{array}}
	\newcommand{\ea}{\end{array}}

\declaretheoremstyle[headfont=\normalfont]{normalhead}

\newtheorem{theorem}{Theorem}[section]
\newtheorem{theoremalph}{Theorem}[section]

\newtheorem{lemma}[theorem]{Lemma}
\newtheorem{corollary}[theorem]{Corollary}
\newtheorem{definition}[theorem]{Definition}

\numberwithin{equation}{section}

\begin{document}
\title{$R$-weighted graphs and commutators}
\author{Harish Kishnani}
\email{harishkishnani11@gmail.com}
\address{Indian Institute of Science Education and Research Mohali, Sector 81, Mohali 140306, India}

\author{Amit Kulshrestha}
\email{amitk@iisermohali.ac.in}
\address{Indian Institute of Science Education and Research Mohali, Sector 81, Mohali 140306, India}

\thanks{The authors thank the DST-FIST facility established through grant SR/FST/MS-I/2019/46 to support this research. The first named author acknowledges the support through the Prime Minister's Research Fellowship from the Ministry of Education, Government of India (PMRF ID: 0601097). We are grateful to Anupam Singh for his valuable comments and suggestions.}

\begin{abstract}
In this article, we introduce balance equations over commutative rings $R$ and associate $R$-weighted graphs to them so that solving balance equations corresponds to a consistent labeling of vertices of the associated graph. Our primary focus is the case when $R$ is a commutative local ring whose residue field contains at least three elements. In this case, we provide explicit solutions of balance equations. As an application, letting $R$ to be the ring of $p$-adic integers, we examine some necessary and sufficient conditions for a $p$-group of nilpotency class $2$ to have its set of commutators coincide with its commutator subgroup. We also apply our results to study the surjectivity of the Lie bracket in Lie algebras, without any restriction on their dimension and the underlined field.\\
\end{abstract}
\maketitle

\section{Introduction}

In algebra, commutators occur in several contexts. Some of these are--the commutator word $xyx^{-1}y^{-1}$ on a group $G$, the commutator map $(x,y) \mapsto xy-yx$ on associative rings, and the Lie bracket $[x,y]$ of a Lie algebra. While distinct, these commutators often possess common properties. Thus, one may hope to deal with them together. We make one such attempt in this paper. Though our primary focus is on commutators in groups and Lie algebras, the techniques developed in this paper have applications in wider contexts.
\subsection{Commutators in groups}
Let $G$ be a group, and let $K(G)$ denote the set of commutators of $G$. Let $G'$ be the subgroup generated by $K(G)$. A well-studied classical problem is to determine conditions on $G$ that ascertain $K(G) = G'$. The case of non-abelian finite simple groups is extremely interesting and challenging. Thanks to the work of Ore\cite{Ore-1951}, Ito\cite{ito-1951}, Thompson\cite{Thompson-1961, Thompson-1962(a), Thompson-1962(b)}, Gow\cite{Gow-1988}, Bonten\cite{Bonten-1992}, Neubüser--Pahlings--Cleuvers\cite{Neubüser-Pahlings-Cleuvers-1984}, Ellers--Gordeev\cite{Ellers-Gordeev_1998}, and Liebeck--O'Brien--Shalev--Tiep\cite{Liebeck-O'Brien-Shalev-Tiep-2010}, we now know that $K(G) = G'$ in this case; we refer to \cite{Malle_Ore-Conjecture} for a survey on this.

For the groups that are not simple, this problem is still wide open. An excellent survey in this direction is \cite{Kappe_Morse_2007}. The case of $p$-groups is particularly interesting. It is known that if $G$ is a $p$-group with $p > 3$ and $G'$ has a generating set containing $3$ elements, then $K(G) = G'$ \cite{Guralnick_1982, Heras_2020}. In \cite{Guralnick_1982}, some
examples of groups with $K(G) \neq G'$ are constructed and it is shown that such examples are not possible if $|G| < 96$, or $|G'| < 16$. Further, it is shown that these bounds cannot be improved. 

For $p$-groups, the existing studies often impose the following restrictions.

\begin{enumerate}
\item[(1).] \emph{Restrictions concerning the order of the group}. The problem is well studied for $p$-groups of order at most $p^7$ (see \cite{Kaushik_Yadav_2023}), but not much information is available for groups of higher order.

\item[(2).] \emph{Restrictions concerning the order of the derived subgroup}. There are results for $p$-groups for which the order of $G'$ is at most $p^4$ (see \cite{Kaushik_Yadav_2021}), but not much is known otherwise.

\item[(3).] \emph{Restrictions concerning the size of generating sets of the derived subgroup}. The equality $G' = K(G)$ is known to hold for $p$-groups whose derived subgroup is generated by $3$ elements (see \cite{Guralnick_1982} and \cite{Heras_2020}). However, for other groups, no general results are known in this direction.
\end{enumerate}

In this article, we study this problem for nilpotent groups of class $2$ through a different approach. Let $g \in G'$ and $Z(G)$ be the center of $G$. Given a generating set $B_G$ of the factor group $G/Z(G)$, we express $g$ as a product of commutator powers $[g_i,g_j]^{d_{i,j}}$, where $d_{i,j} \in \mathbb Z$, and $g_i Z(G), g_j Z(G) \in B_G$. Depending on the integers $\{d_{i,j}: i < j\}$, we construct a weighted graph $\Gamma(D)$ that determines whether $g$ is a commutator in $G$ or not. This investigation relies on the concept of \emph{bad cycles} in non-weighted graphs and
\emph{unfavorable proximity} of bad cycles in a weighted graph. These notions are introduced in Definitions \ref{bad cycle} and \ref{unfavorable proximity}. The graph $\Gamma(D)$ depends not only on $g$, but on the choices $d_{i,j} \in \mathbb Z$ with 
$g = \prod_{i < j}[g_i,g_j]^{d_{i,j}}$ too.

Our main results are as follows.

\begin{theoremalph}\label{theorem in intro for g being a commutator, arbitrary p}
Let $p \ne 2$ and $G$ be a $p$-group of nilpotency class $2$. Let $g = \prod_{i < j}[g_i,g_j]^{d_{i,j}} \in G'$, where $g_iZ(G), g_jZ(G) \in B_G$, be such that the graph $\Gamma(D)$ does not contain bad cycles. Then $g$ is a commutator in $G$. (Theorem \ref{without bad cycle for p-groups elementwise})
\end{theoremalph}

\begin{theoremalph}\label{theorem in intro for g not a commutator}
Let $G$ be a $p$-group of nilpotency class $2$. Let $g \in G'$ be such that the graph $\Gamma(D)$ contains a bad cycle with unfavorable proximity for each choice $\{d_{i,j} : i < j\}$ with $g= \prod_{i < j}[g_i,g_j]^{d_{i,j}}$; $g_iZ(G), g_jZ(G) \in B_G$. Then $g$ is not a commutator in $G$. (Theorem \ref{Not in the image of commutator map on p-groups})
\end{theoremalph}

These results are independent of the above restrictions and hold for some infinite $p$-groups of nilpotency class $2$, too. We use them to construct groups with $K(G) = G'$ and $K(G) \neq G'$.

For a $p$-group $G$ of nilpotency class $2$, with a finite generating set $B_G$ of $G/Z(G)$, we associate a (non-weighted) graph $\Gamma(B_G)$. If $p$ is odd, the Corollary \ref{Without bad cycle for whole group} guarantees $K(G) =G'$, provided $\Gamma(B_G)$ does not contain bad cycles. With some essential assumptions, Theorem \ref{iff condition for p-groups} establishes the converse, and thus we obtain a characterization of $p$-groups of nilpotency class $2$ with $K(G) = G'$. As an application, in Corollary \ref{Infinite examples for K(G) not equal to G'} we construct infinitely many groups with $K(G) \ne G'$. A weaker version of Theorem \ref{theorem in intro for g being a commutator, arbitrary p} and Corollary \ref{Without bad cycle for whole group} is proved for $p=2$ case in Theorem \ref{Arbitrary p-groups elementwise} and Corollary \ref{Arbitrary $p$-group}, respectively. Remark \ref{Remark showing p=2 case is same for small groups} shows that the assumption $p \ne 2$ can be dropped in Theorem \ref{theorem in intro for g being a commutator, arbitrary p} and Corollary \ref{Without bad cycle for whole group} for groups of small orders. 

Our core idea lies in solving a system of balance equations on local rings and expressing it in terms of a consistent labeling on graphs constructed out of our groups, in such a way that the existence of consistent labeling on these graphs corresponds to $K(G) = G'$. This idea extends to a study of the surjectivity of alternating bilinear maps as well, and in particular, to the study of commutators in Lie algebras. 


\subsection{Commutators in Lie algebras}
Let $L$ be a Lie algebra over a field $F$ and let $L'$ be its derived subalgebra. Let $[L,L]:=\{[x,y]: x,y \in L\}$. It has been a problem of great interest to determine the cases when $[L,L] = L'$. It is clear that if $[L,L] = L$, then $[L,L]=L'$.
In \cite{Brown_1963}, Brown proved that $[L,L] = L$ for any finite-dimensional complex simple Lie algebra. Akhiezer extended this result to most of the finite-dimensional simple real Lie algebras \cite{Akhiezer_2015}, but the problem is still open in this case for an arbitrary finite-dimensional simple real Lie algebra. For a finite-dimensional nilpotent Lie algebra with $\dim (L') \leq 4$, this problem has been investigated in \cite{Niranjan-Rani_2023}. In \cite{DKR_2021} and \cite{KMR_2024}, the authors have given examples of infinite-dimensional simple Lie algebras with $[L,L] \ne L$. Our results on consistent labeling apply to all Lie algebras without any restriction on their dimension and the underlined field. 

Let $x \in L'$. Given a basis of $L/Z(L)$, we express $x$ as a linear sum of the Lie bracket in elements of this basis. Let $x = \sum_{1 \leq i < j \leq r}{d_{i,j}}[u_i,u_j]$, where $d_{i,j} \in \mathbb F$ and $u_i, u_j$ are the basis elements of $L/Z(L)$. Depending on the scalers $\{d_{i,j}: i < j\}$, we construct a weighted graph $\Gamma(D)$ that determines whether $x \in [L,L]$ or not. The graph $\Gamma(D)$ depends not only on $x$, but also on the choices $d_{i,j} \in \mathbb F$ with $x = \sum_{1 \leq i < j \leq r}{d_{i,j}}[u_i,u_j]$. Using balance equations and consistent labeling on the corresponding graphs, we obtain the following results.

\begin{theoremalph}
Let $F \neq \mathbb{F}_2$ be a field and $L$ be a Lie algebra over $F$ having a countable Hamel basis. Let $x = \sum_{1 \leq i < j \leq r}{d_{i,j}}[u_i,u_j] \in L'$ be such that the graph $\Gamma(D)$ does not contain bad cycles. Then $x \in [L,L]$.
(Theorem \ref{without bad cycle for Lie algebras})
\end{theoremalph}

\begin{theoremalph}
Let $F$ be a field and $L$ be a Lie algebra over $F$ having a countable Hamel basis. Let $x \in L'$ be such that the graph $\Gamma(D)$ contains a bad cycle with unfavorable proximity for each choice $\{d_{i,j}:i < j\}$ with $x = \sum_{1 \leq i < j \leq r}{d_{i,j}}[u_i,u_j]$. Then $x \notin [L,L]$.
(Theorem \ref{Not in the image of Lie bracket})
\end{theoremalph}

For Lie algebras with finite-dimensional quotient Lie algebra $L/Z(L)$, we associate a (non-weighted) graph $\Gamma(\mathcal B_L)$. If $F \ne \mathbb F_2$, then Corollary \ref{Without bad cycle for whole Lie algebra} guarantees $[L,L] =L'$, provided $\Gamma(\mathcal B_L)$ does not contain bad cycles. With some essential assumptions, Theorem \ref{iff condition for Lie algebra} establishes the converse, and thus we obtain a characterization of Lie algebras with $[L,L] =L'$. Similar to our construction for $p$-groups of nilpotency class $2$, one can construct infinitely many Lie algebras with $[L,L]  \ne L'$.

The organization of the paper is as follows. In \S\ref{results on graphs}, we focus on some graphs, their step-by-step constructions, and properties that are relevant to our context. In $\S \ref{Labeling of Graphs}$, we introduce a system of balance equations over a commutative ring $R$. This system consists of equations $x_iy_j - x_jy_i = d_{i,j}$, where $d_{i,j} \in R$ and $1 \leq i < j \leq n$. We associate a weighted graph to a system of balance equations and show that its solution corresponds to a consistent labeling of vertices of this graph. We use this approach to obtain solutions of the system when $R$ is a local ring whose residue field contains at least three elements. As an application, in \S\ref{section for results on p-groups} and \S\ref{section on bilinear maps and Lie algebra}, we obtain the main results of this paper.

Although the goal of this article is to understand commutator word on groups, and the Lie bracket on Lie algebras, the machinery developed in the process is interesting on its own. We have introduced the idea of borderless graphs and nets in $\S \ref{results on graphs}$ and have proved several interesting results. This includes providing an iterative construction of these graphs and proving some structural results. For graphs that are free from bad cycles, we have also introduced the notion of an anchor of a graph. This enables us to construct a sign function on the vertex set of a net that does not contain bad cycles. These results are used together in $\S \ref{Labeling of Graphs}$ to obtain a consistent labeling on graphs that do not contain bad cycles.

A word on convention and notation in this article -- we assume that all graphs are finite, simple and connected. Unless specified otherwise, the vertex set of a graph $\Gamma(V,E)$ is $V:= \{v_1, v_2, \cdots, v_n\}$, and the edge set is $E := \{e_{i,j}\}$, where $e_{i,j}$ denotes the edge between the vertices $v_i$ and $v_j$. We denote the degree of $v_i$ in $\Gamma$ by $\deg_{\Gamma}(v_i)$. A \emph{path} in $\Gamma$ is a sequence
$v_{i_1} \to v_{i_2} \to \cdots \to v_{i_{s-1}} \to v_{i_s}$, where each $v_{i_k} \in V$ and
$e_{i_k, i_{k+1}} \in E$. A path is \emph{simple} if the vertices occurring in it are all distinct.
A path $v_{i_1} \to v_{i_2} \to \cdots \to v_{i_{r}} \to v_{i_{r+1}}$ is called an \emph{$r$-cycle} if $i_1 = i_{r+1}$. We use the notation $C_r$ to denote an $r$-cycle. A cycle $v_{i_1} \to v_{i_2} \to \cdots \to v_{i_{r}} \to v_{i_{r+1}}$ is called \emph{simple} if $v_{i_1} \to v_{i_2} \to \cdots \to v_{i_{r}}$ is a simple path.


\section{Some graphs and their properties}\label{results on graphs}
Let $\Gamma_i = (V_i, E_i)$; $i = 1,2$, be two graphs. For a given $(v, w) \in {V}_1 \times {V}_2$, we define the \emph{wedge sum} 
$\Gamma_1 \bigwedge_{v=w} \Gamma_2 = \Gamma(V,E)$ as follows.

\begin{enumerate}
\item[$(i)$.] $V := (V_1 \bigsqcup V_2) \setminus \{w\}$,
where $V_1 \bigsqcup V_2$ is the disjoint union of $V_1$ and $V_2$.
\item[$(ii)$.] $E_1 \subseteq E$.
\item[$(iii)$.] All edges in $E_2$ that are not incident on $w$ belong to $E$. If $u \in V_2$ is adjacent to $w$ in $\Gamma_2$, then there is an edge in $E$ between $u$ and $v$.
\end{enumerate}

The graph $\Gamma_1 \bigwedge_{v=w} \Gamma_2$ is called the \textit{wedge sum} of $\Gamma_1$ and $\Gamma_2$ along the vertices $v$ and $w$. Equivalently, we say that $\Gamma$ is obtained by \emph{gluing} $\Gamma_1$ and $\Gamma_2$ along $v$ and $w$.

\subsection{Borderless Graphs}
A finite connected graph $\Gamma$ is said to be a \textit{borderless} graph if any pair of distinct cycles in $\Gamma$ have at most one common vertex. A tree and a graph with exactly one cycle are obvious examples of such graphs. It is evident that a subgraph of a borderless graph is borderless. We establish some properties of borderless graphs in this subsection. These properties will be used in \S \ref{Labeling of Graphs} to provide a consistent labeling for these graphs.

\begin{lemma}\label{two vertex lemma for borderless}
Let $\Gamma$ be a borderless graph and $C_m$ be a simple cycle in $\Gamma$. Let $u$ and $v$ be any two vertices of $C_m$. Then any simple path connecting $u$ and $v$ in $\Gamma$ lies entirely in $C_m$.
\end{lemma}

\begin{proof}
On the contrary, let $P_1: u = u_0\to u_1 \to \cdots \to u_{r-1} \to u_r = v$ be a simple path in $\Gamma$ joining $u$ and $v$ that does not lie entirely in $C_m$. Let $j \geq i$ be such that the vertices $u_i$, $u_{i+1}, \cdots, u_j$ lie outside the cycle $C_m$ but $u_{i-1}$ and $u_{j+1}$ are in $C_m$. Since $u_{i-1}$ and $u_{j+1}$ lie in $C_m$, there exists a path $P_2$ between $u_{i-1}$ and $u_{j+1}$ that lies entirely in $C_m$. Then the path $P_1': u_{i-1}\to u_i \to \cdots \to u_{j} \to u_{j+1}$ followed by $P_2$ is a cycle, which is different from $C_m$ but has more than one common vertex. This contradicts the hypothesis that $\Gamma$ is borderless.
\end{proof}

The following lemma gives crucial structural information about borderless graphs.

\begin{lemma}\label{pendant or borderless cycle}
Let $\Gamma$ be a borderless graph. Then one of the following holds.
\begin{enumerate}
\item $\Gamma$ is a simple cycle.
\item $\Gamma$ has a pendant vertex.
\item $\Gamma$ properly contains a simple cycle which has a unique vertex of degree greater than $2$ in $\Gamma$.
\end{enumerate}
\end{lemma}

\begin{proof}
  Let $r$ be the number of simple cycles in $\Gamma$. We proceed by induction on $r$. If $r = 0$, then $\Gamma$ is a tree and thus it has a pendant vertex. So, we assume that $r \geq 1$, and that the lemma holds for all borderless graphs containing less than $r$ cycles. If $\Gamma$ is a simple cycle, then the lemma follows. Thus, we assume that $\Gamma$ properly contains a simple cycle.

\noindent{\bfseries Step 1}. Let $C_{r_1}$ be such a cycle. If it has a unique vertex of degree greater than $2$ in $\Gamma$, then the lemma follows. So, let $u_1$ and $u_2$ be two vertices of $C_{r_1}$ such that $\deg(u_1) > 2$ and $\deg(u_2) > 2$. By Lemma \ref{two vertex lemma for borderless}, we conclude that any path between $u_1$ and some other vertex of $C_{r_1}$ lies entirely in $C_{r_1}$. Thus, there exist two subgraphs $\Gamma_1$ and $\Gamma_2$ of $\Gamma$ such that $\Gamma = \Gamma_1 \bigwedge_{u_1 = v} \Gamma_2$, $C_{r_1} \subsetneq \Gamma_1$ and $\deg(u_1)$ in $\Gamma_1$ is $2$. 

\noindent{\bfseries Step 2}. Clearly, the number of cycles in $\Gamma_2$ is less than $r$ and we apply induction on $\Gamma_2$. If $\Gamma_2$ is a tree, then it has at least two pendant vertices. So, $\Gamma$ has at least one pendant vertex. If $\Gamma_2$ is a simple cycle, then it is a simple cycle in $\Gamma$ which has a unique vertex $u_1$ of degree greater than $2$ in $\Gamma$, and thus the lemma follows. So, we assume that $\Gamma_2$ properly contains a simple cycle. Let it be $C_{r_2}$. If $C_{r_2}$ has a unique vertex with a degree greater than $2$ in $\Gamma$, then the lemma follows. So, let $u_3$ and $u_4$ be two distinct vertices of $C_{r_2}$ which have degree greater than $2$ in $\Gamma$. By Lemma \ref{two vertex lemma for borderless}, there exist two subgraphs $\Gamma_3$ and $\Gamma_4$ of $\Gamma$ such that $\Gamma = \Gamma_3 \bigwedge_{u' = v'} \Gamma_4$, $C_{r_2} \subsetneq \Gamma_3$, $\Gamma_1 \subsetneq \Gamma_3$, $u'$ is either $u_3$ or $u_4$, and $\deg(u')$ in $\Gamma_3$ is $2$. 

\noindent{\bfseries Step 3}. Repeat Step $2$ after replacing $\Gamma_2$ by $\Gamma_4$. Then we either get a pendant vertex in $\Gamma$, or a simple cycle with a unique vertex of degree greater than $2$ in $\Gamma$, or there exists a simple cycle $C_{r_3}$ in $\Gamma_4$ which has at least two vertices $u_5$ and $u_6$ with degree greater than $2$ in $\Gamma$. Hence, there exist two subgraphs $\Gamma_5$ and $\Gamma_6$ of $\Gamma$ such that $\Gamma = \Gamma_5 \bigwedge_{u'' = v''} \Gamma_6$, $C_{r_3} \subsetneq \Gamma_5$, $\Gamma_3 \subsetneq \Gamma_5$, $u''$ is either $u_5$ or $u_6$, and $\deg(u'')$ in $\Gamma_5$ is $2$.

We again repeat Step $2$ after replacing $\Gamma_2$ by $\Gamma_6$. Note that since $\Gamma$ is a finite graph, we cannot relentlessly repeat this step. Thus, the lemma follows after a finite repetition of Step $2$.
\end{proof}


The following lemma provides a mechanism for iteratively constructing a borderless graph.

\begin{lemma}\label{inductive-gluing-for-borderless}
Let $\Gamma$ be a borderless graph. Then $\Gamma$ can be constructed iteratively, by gluing either a tree or a cycle at each step.
\end{lemma}

\begin{proof}
The proof is by induction on $m$, the number of edges in $\Gamma$. If $m \leq 3$, then $\Gamma$ is either a tree or a cycle and the lemma holds trivially. Now, let $m>3$. We may assume that $\Gamma$ properly contains a cycle, otherwise $\Gamma$ is either a tree or a cycle. We consider the following two cases.

{\bfseries Case 1}. \quad {\it $\Gamma$ contains a pendant vertex}. \\
Let $v$ be a pendant vertex and $u$ be the vertex adjacent to $v$ in $\Gamma$. Let $\Gamma_1$ be the graph obtained after deleting the edge joining $u$ and $v$ from $\Gamma$. By induction, the lemma holds for $\Gamma_1$. Since we can obtain $\Gamma$ from $\Gamma_1$ by gluing an edge, it holds for $\Gamma$ as well. 

{\bfseries Case 2}.  \quad {\it $\Gamma$ does not contain a pendant vertex}. \\
Since $\Gamma$ is borderless, and does not contain a pendant vertex, by Lemma \ref{pendant or borderless cycle}, it must contain a cycle $C_r$ which has a unique vertex of degree greater than $2$ in $\Gamma$. Let $\Gamma_1$ be the graph obtained after deleting the cycle $C_r$ from $\Gamma$. Then $\Gamma_1$ is borderless. By the induction, the lemma holds for $\Gamma_1$ and thus, it holds for $\Gamma$ as well.
\end{proof}

For a borderless graph $\Gamma$, we define its  {\it height} $s(\Gamma)$ to be the minimal number of steps required to construct $\Gamma$ by gluing either a cycle or a tree at each step. In the base case when $\Gamma$ is a tree or a cycle we set $s(\Gamma) = 1$. In Figures \ref{Borderless graph 1} and \ref{Borderless graph 2}, we illustrate two borderless graphs of height $2$.

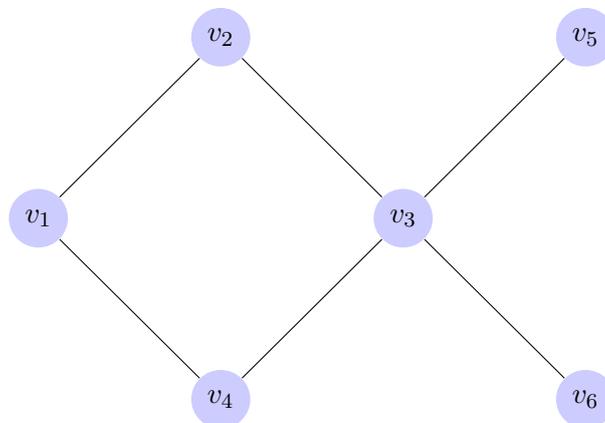
\begin{figure}[h]
\begin{tikzpicture}
  [scale=.6,auto=left,every node/.style={circle,fill=blue!20}]
  \node (n1) at (1,8)  {$v_1$};
  \node (n2) at (5,12)  {$v_2$};
  \node (n3) at (9,8)  {$v_3$};
  \node (n4) at (5,4)  {$v_4$}; 
  \node (n5) at (13,12)  {$v_5$};
  \node (n6) at (13,4)  {$v_6$};

  \foreach \from/\to in {n1/n2,n2/n3,n3/n4,n4/n1,n3/n5,n3/n6}
      \draw (\from) -- (\to);
\end{tikzpicture}
\caption{A borderless graph $\Gamma_1$ with $s(\Gamma_1) = 2$.
\label{Borderless graph 1}
}
\end{figure}

\begin{figure}
\begin{tikzpicture}
  [scale=.6,auto=left,every node/.style={circle,fill=blue!20}]
  \node (n1) at (1,4)   {$v_1$};
  \node (n2) at (1,8)   {$v_2$};
  \node (n3) at (5,12)  {$v_3$};
  \node (n4) at (9,8)   {$v_4$};
  \node (n5) at (9,4)   {$v_5$};
  \node (n7) at (17,8)  {$v_7$};
  \node (n6) at (13,12) {$v_6$};
  \node (n8) at (13,4)  {$v_8$};
  
  \foreach \from/\to in {n1/n2,n2/n3,n3/n4,n1/n5,n4/n5,n4/n6,n4/n8,n6/n7,n7/n8}
      \draw (\from) -- (\to);
\end{tikzpicture}
\caption{A borderless graph $\Gamma_2$ with $s(\Gamma_2) = 2$.
\label{Borderless graph 2}
}
\end{figure}
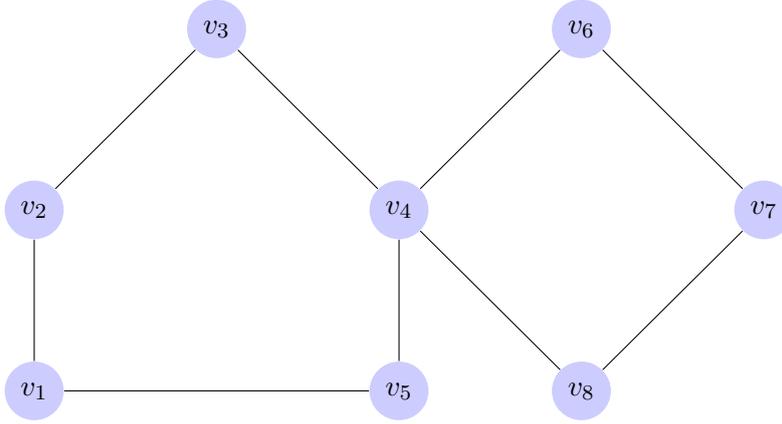

\subsection{Nets}
A graph $\Gamma$ with the vertex set $V = \{ v_1, v_2, \dots , v_n\}$ and the edge set $E = \{ e_1, e_2, \dots , e_{n-1} \}$ is called a \textit{segment}, if the vertices $v_i$ and $v_{i+1}$ are joined by the edge $e_i$ for each $i \in \{ 1, 2, \dots , n-1\}$. We define the vertices of degree one of a segment as \textit{end points} of the segment. We call a subgraph $\Gamma_1$ of a graph $\Gamma$, a \textit{maximal segment} of $\Gamma$ if it has the following properties.
\begin{enumerate}
    \item $\Gamma_1$ is a segment.
    \item The end points of $\Gamma_1$ have degree different from $2$ in $\Gamma$.
    \item All vertices of $\Gamma_1$, except the end points, have degree $2$ in $\Gamma$.
\end{enumerate}

A connected finite graph $\Gamma$ is called a \textit{net} if it can be constructed by taking a simple cycle in the first step and then iteratively gluing the endpoints of a segment to two different points at each step to follow. For a net $\Gamma$, we define $\eta(\Gamma)$ as the minimum number of steps needed to iteratively construct $\Gamma$. If $\Gamma$ is a simple cycle, then clearly $\eta(\Gamma) = 1$. The graph $\Gamma_1$ in Figure \ref{eta(Gamma) = 2} has $\eta(\Gamma_1) = 2$. 

\begin{center}
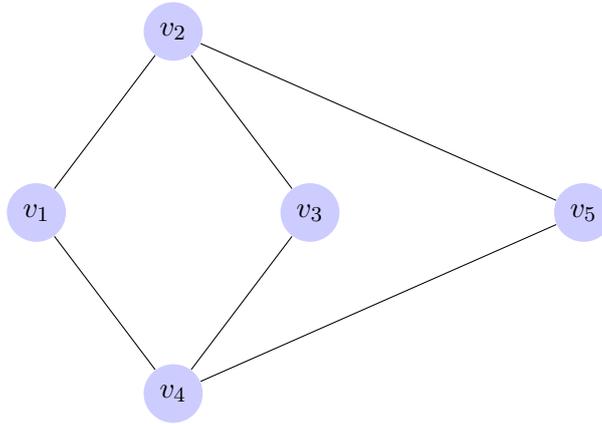
\begin{figure}[h]
\begin{tikzpicture}
  [scale=.6,auto=left,every node/.style={circle,fill=blue!20}]
  \node (n1) at (1,8)  {$v_1$};
  \node (n2) at (4,12)  {$v_2$};
  \node (n3) at (7,8)  {$v_3$};
  \node (n4) at (4,4)  {$v_4$}; 
  \node (n5) at (13,8) {$v_5$};

  \foreach \from/\to in {n1/n2,n2/n3,n3/n4,n4/n1,n2/n5,n5/n4}
      \draw (\from) -- (\to);
\end{tikzpicture}
\caption{Example of a net $\Gamma_1$ with $\eta(\Gamma_1) =2$.}\label{eta(Gamma) = 2}
\end{figure}
\end{center}
The following lemma shows that $\eta(\Gamma)$ is independent of the iterative construction of $\Gamma$.
\begin{lemma}\label{eta(Gamma) equation}
Let $\Gamma$ be a net with vertex set $V$ and edge set $E$. Then $\eta(\Gamma) = |E| -|V| + 1$.
\end{lemma}

\begin{proof}
We prove it by induction on $\eta(\Gamma)$. If $\eta(\Gamma) = 1$, then $\Gamma$ is a simple cycle and thus $|E| -|V| + 1 = 1$. If $\eta(\Gamma) =: r >1$, then by the induction hypothesis, we assume that the lemma holds for all $\Gamma$ with $\eta(\Gamma) < r$. Let $\Gamma_1$ be the segment glued to the net $\Gamma_2$ in the last step of the iterative construction of $\Gamma$. Let $\Gamma_1$ has vertex set $V_{\Gamma_1}$ and edge set $E_{\Gamma_1}$, and $\Gamma_2$ has vertex set $V_{\Gamma_2}$ and edge set $E_{\Gamma_2}$. Since $\eta(\Gamma_2) = r-1$, by induction $\eta(\Gamma_2) = |E_{\Gamma_2}|  - |V_{\Gamma_2}| + 1 = r-1$. Further, as the end points of $\Gamma_1$ are identified with two vertices of $\Gamma_2$ to obtain $\Gamma$, thus $|V| = |V_{\Gamma_2}| + |V_{\Gamma_1}|-2$ and $|E| = |E_{\Gamma_2}| + |E_{\Gamma_1}|$. Furthermore, since $\Gamma_1$ is a segment, $|V_{\Gamma_1}| = |E_{\Gamma_1}| + 1$. Thus,

\begin{align*}
|E| - |V| + 1 &= |E_{\Gamma_2}| + |E_{\Gamma_1}| -|V_{\Gamma_2}| - |V_{\Gamma_1}| + 2  + 1 \\
&= \left( |E_{\Gamma_2}|  - |V_{\Gamma_2}| + 1 \right) +1 \\
&=\eta(\Gamma_2) + 1 \\
&= \eta(\Gamma).
\end{align*}
\end{proof}

\begin{lemma}\label{deleting maximal segments from a net}
 Let $\Gamma$ be a net with $\eta(\Gamma) \geq 2$. Let $\Gamma'$ be a subgraph of $\Gamma$, obtained by deleting the edges of a maximal segment in $\Gamma$. Then $\Gamma'$ is a net with $\eta(\Gamma') = \eta(\Gamma) - 1$.
\end{lemma}

\begin{proof}
We proceed by induction on $\eta(\Gamma)$. If $\eta(\Gamma) = 2$, then $\Gamma$ contains three maximal segments. Removal of any of these maximal segments from $\Gamma$ gives a simple cycle and thus the lemma follows. Now, suppose $r : = \eta(\Gamma) >2$ and the lemma holds for all nets $\Gamma'$ with $\eta(\Gamma') < r$. Let $\Gamma_1$ be a segment, and $\Gamma_2$ be a net having $\eta(\Gamma_2) = r-1$, such that after gluing $\Gamma_1$ at two distinct points of $\Gamma_2$, we obtain $\Gamma$. Let $\Gamma_3$ be a maximal segment, removed from $\Gamma$ to get a graph $\Gamma_4$. We show that $\Gamma_4$ is a net. For that, we first assume that $\Gamma_3$ is the same as $\Gamma_1$. In this case, it turns out that $\Gamma_4$ is $\Gamma_2$, which is a net. Now, if $\Gamma_3$ is a maximal segment different from $\Gamma_1$, then it is contained in $\Gamma_2$. By induction on $\Gamma_2$, the graph obtained after removing the edges of the maximal segment $\Gamma_3$ from $\Gamma_2$ is a net. Now, if we glue the segment $\Gamma_1$ to it, then we get $\Gamma_4$. Thus, $\Gamma_4$ is also a net. The fact that $\eta(\Gamma_4) = \eta(\Gamma)-1$ follows directly by Lemma \ref{eta(Gamma) equation}.
\end{proof}

The following lemma gives an equivalent definition for nets.

\begin{lemma}\label{net equivalent definition}
Let $\Gamma$ be a finite connected graph. Then $\Gamma$ is a net with $\eta(\Gamma) \geq 2$ if and only if it has the following properties.

\begin{enumerate}
    \item $\Gamma$ contains at least two simple cycles.
    \item Each edge of $\Gamma$ is a part of some simple cycle in $\Gamma$.
    \item For each simple cycle $C$ in $\Gamma$, there exists a simple cycle $C'$ in $\Gamma$, different from $C$, such that $C$ and $C'$ share a common maximal segment in $\Gamma$.
\end{enumerate}
\end{lemma}

\begin{proof}
We first assume that $\Gamma$ is a net with $\eta(\Gamma) \geq 2$. Clearly, it contains at least two simple cycles. 
To show that $\Gamma$ has the rest of the properties, we apply induction on $\eta(\Gamma)$. Suppose $\eta(\Gamma) = 2$.
Then these properties are evident. Now, suppose $\eta(\Gamma) = r \geq 3$. By induction hypothesis, the conclusion of the lemma holds for all nets $\Gamma'$ with $\eta(\Gamma') < r $. Let $\Gamma_1$ be a segment, and $\Gamma_2$ be a net having $\eta(\Gamma_2) = r-1$, such that after gluing $\Gamma_1$ at two distinct points (say $u$ and $w$) of $\Gamma_2$, we obtain $\Gamma$. By induction, the lemma holds for $\Gamma_2$. 

We first show that $\Gamma$ satisfies $(2)$. Since $u, w \in \Gamma_2$ and $\Gamma_2$ is connected, there exists a simple path $P_1$ between $u$ and $w$ in $\Gamma_2$. Thus, $P_1$ along with the edges of $\Gamma_1$ forms a simple cycle, $C^{(1)}$ in $\Gamma$. Hence, all the edges of $\Gamma$ are part of some simple cycle in $\Gamma$. Now, we show that $\Gamma$ satisfies $(3)$. Let $C$ be a simple cycle in $\Gamma$. If $C \subseteq \Gamma_2$, then by induction there exists a simple cycle $C^{(2)}$ in $\Gamma_2$, and thus in $\Gamma$, different from $C$, such that $C$ and $C^{(2)}$ shares a common maximal segment in $\Gamma$. If $C \not\subset \Gamma_2$, then all the edges of $\Gamma_1$ are part of $C$. Further, if $C \neq C^{(1)}$, then all the edges of $\Gamma_1$ are common in $C$ and $C^{(1)}$, and hence the lemma follows in this case as well. Finally, suppose $C^{(1)} = C$. Let $\Gamma_3$ be a maximal segment common in $C$ and $\Gamma_2$. By induction, the edges of $\Gamma_3$ are contained in some simple cycles and since $\Gamma_3$ is a maximal segment in $\Gamma$, there exists a simple cycle $C^{(3)}$ in $\Gamma_2$ which contains all the edges of $\Gamma_3$. Thus, $C$ and $C^{(3)}$ share a common maximal segment $\Gamma_3$ in $\Gamma$.

Now, let us assume that $\Gamma$ satisfies $(1)$, $(2)$, $(3)$. We aim to show that $\Gamma$ is a net with $\eta(\Gamma) \geq 2$. By $(1)$, it is enough to show that $\Gamma$ is a net. Denote $r(\Gamma) := |E| -|V|$. By $(2)$, the degree of each vertex in $\Gamma$ is at least $2$. Thus, $r(\Gamma) \geq 0$. If $r(\Gamma) = 0$, then the degree of each vertex is $2$. So, by $(2)$, $\Gamma$ is a simple cycle. This contradicts $(1)$. 

If $r(\Gamma) = 1 $, then $|E| = \frac{1}{2}\sum_{v \in V} \deg(v) = |V| + 1$. Thus, either $\Gamma$ has a unique vertex $u$ of degree greater than two with $\deg(u)=4$, or $\Gamma$ has precisely two vertices of degree $3$. We show that the former case does not arise.

To the contrary, suppose that $u$ is a unique vertex of degree $4$ in $\Gamma$ and there is no other vertex with a degree greater than $2$. Let $v$ be a vertex adjacent to $u$ in $\Gamma$. By $(2)$, the edge between $u$ and $v$ is contained in some simple cycle of $\Gamma$, say $C^{(1)}$. By $(3)$, there exists a simple cycle $C^{(2)}$ in $\Gamma$, different from $C^{(1)}$, such that $C^{(1)}$ and $C^{(2)}$ shares a common maximal segment in $\Gamma$. Since the end points of this maximal segment are of a degree greater than $2$, we have a contradiction. 

Thus, $\Gamma$ has precisely two vertices of degree $3$, say $u$ and $v$ and there is no other vertex with a degree greater than $2$. By (3), each simple cycle in $\Gamma$ contains at least two vertices of degree greater than $2$. Thus, each simple cycle in $\Gamma$ necessarily contains $u$ and $v$. Let $C^{(1)}$ be a simple cycle containing $u$ and $v$. By $(3)$, there exists a simple cycle $C^{(2)}$ in $\Gamma$ which shares a maximal segment with $C^{(1)}$. The endpoints of this segment must be $u$ and $v$.

Note that any vertex of $\Gamma$ is contained either in $C^{(1)}$ or $C^{(2)}$. To see this, suppose there exists $w \in V$ which is not in the vertex set of $C^{(1)} \bigcup C^{(2)}$.
Then there is a simple path terminating at a vertex $w'$ of $C^{(1)} \bigcup C^{(2)}$ and has edges lying outside $C^{(1)} \bigcup C^{(2)}$. Thus, $\deg(w') \geq 3$, and hence $w' \in \{u, v\}$. Thus, either $\deg(u) > 3$ or $\deg(v) > 3$, which is a contradiction.

Finally, we conclude that $\Gamma = C^{(1)} \bigcup C^{(2)}$, which is a net
with $\eta(\Gamma) = 2$.

If $r(\Gamma) \geq 2$, then we proceed by induction and assume that the lemma holds for every finite connected graph $\Gamma'$ with $r(\Gamma') < r(\Gamma)$. Let $C^{(1)}$ be a cycle in $\Gamma$. By $(3)$, there exists a cycle $C^{(2)}$ in $\Gamma$ which shares a maximal segment, say $\Gamma_1$ with $C^{(1)}$. Let $\Gamma_2$ be the subgraph obtained by deleting the edges of $\Gamma_1$. It is clear that $r(\Gamma_2) = r(\Gamma) - 1$. By induction, $\Gamma_2$ is a net, and hence $\Gamma$ is also a net.   
\end{proof}

\begin{figure}[h]
\begin{tikzpicture}
  [scale=.6,auto=left,every node/.style={circle,fill=blue!20}]
  \node (n1) at (1,8)  {$v_1$};
  \node (n2) at (4,12)  {$v_2$};
  \node (n3) at (7,8)  {$v_3$};
  \node (n4) at (4,4)  {$v_4$}; 
  \node (n5) at (13,8) {$v_5$};
  \node (n6) at (19,8) {$v_6$};

  \foreach \from/\to in {n1/n2,n2/n3,n3/n4,n4/n1,n2/n5,n5/n4,n6/n5}
      \draw (\from) -- (\to);
\end{tikzpicture}
\caption{Example of a graph that is not net, but satisfies conditions (1) and (3) of Lemma \ref{net equivalent definition}. \label{2-is-non-redundant}}
\end{figure}

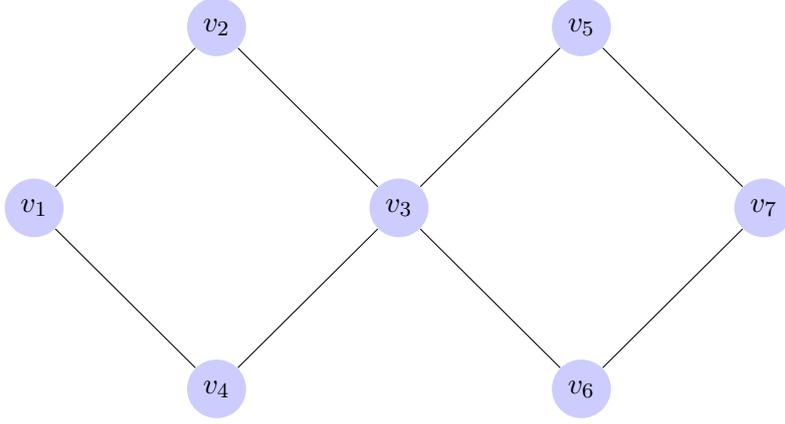
\begin{figure}
\begin{tikzpicture}
  [scale=.6,auto=left,every node/.style={circle,fill=blue!20}]
  \node (n1) at (1,8)  {$v_1$};
  \node (n2) at (5,12)  {$v_2$};
  \node (n3) at (9,8)  {$v_3$};
  \node (n4) at (5,4)  {$v_4$}; 
  \node (n5) at (13,12)  {$v_5$};
  \node (n6) at (13,4)  {$v_6$};
  \node (n7) at (17,8)  {$v_7$};

  \foreach \from/\to in {n1/n2,n2/n3,n3/n4,n4/n1,n3/n5,n3/n6,n5/n7,n6/n7}
      \draw (\from) -- (\to);
\end{tikzpicture}
\caption{Example of a graph that is not net, but satisfies conditions (1) and (2) of Lemma \ref{net equivalent definition}.
\label{3-is-non-redundant}
}
\end{figure}

\begin{remark}
 None of the conditions in Lemma \ref{net equivalent definition} is redundant. To see this, first observe that the condition $(1)$ is needed to ensure that $\eta(\Gamma) \geq 2$. Further, the graphs in Figure \ref{2-is-non-redundant} and Figure \ref{3-is-non-redundant} demonstrate that the conditions $(2)$ and $(3)$ are also non-redundant.
\end{remark}


We now define the notion of bad cycles, admissible sets, and anchors in a graph. 
\begin{definition}\label{bad cycle}
A simple cycle $C_r: u_{1} \to u_{2} \to \dots \to u_{r} \to u_1$ in a graph $\Gamma$ is said to be a \textit{bad cycle} if $\deg_{\Gamma}(u_{i}) > 2$ for $i \leq r-2$; and $\deg_{\Gamma}(u_{r-1})= \deg_{\Gamma}(u_{r}) = 2$.
\end{definition}

A set $W := \{u_1, u_2, \dots, u_k \} \subseteq V$ is called {\it admissible} in $\Gamma$, if the supergraph of $\Gamma$ formed by adding an edge from each $u_i \in W$ to a new vertex $u_i' \notin V$ (where $u_i' \neq u_j'$ for $i \neq j$) does not contain a bad cycle. A maximal admissible set is called an \textit{anchor} in $\Gamma$.

A graph may have multiple or no anchors. However, we note that the set
$$\mathcal P := \{v \in V : \deg_{\Gamma}(v) \neq 2 \text{ or no simple cycle contains } v\}$$ is a subset of all anchors. For a given anchor $\mathcal A$, the elements of $\mathcal A$ will be called {\it $\mathcal A$-points} and the rest of the elements of $V$ will be called {\it non $\mathcal A$-points}.


We make an observation: Let $\Gamma \ne C_3$ be a connected graph that does not contain a bad cycle. Let $\mathcal A$ be an anchor in $\Gamma$. Then a vertex $v \in \Gamma$ is a non $\mathcal A$-point if and only if there exists a cycle $C_r$ in $\Gamma$ containing $v$ whose all but two vertices are non $\mathcal A$-points.

The following lemma provides an iterative construction of a net.

\begin{lemma}\label{inductive construction net}
Let $\Gamma$ be a net that is not a triangle and does not contain bad cycles. Let $\mathcal{A}$ be an anchor for $\Gamma$. Then $\Gamma$ can be constructed iteratively, by taking a simple cycle having two non $\mathcal{A}$-points in the first step, and gluing the endpoints of a segment to two different points at each step to follow, in such a way that the segment to be glued has either one or two non $\mathcal{A}$-points. 
\end{lemma}
\begin{proof}
We prove it by induction on $\eta (\Gamma)$. If $\eta (\Gamma) =1$, then it follows trivially. Now, suppose $r := \eta(\Gamma) \geq 2$ and that the lemma holds for all nets $\Gamma'$ with $\eta(\Gamma') < r$. Let $\Gamma_1$ be a segment, and $\Gamma_2$ be a net with $n(\Gamma_2) = r-1$, such that after gluing $\Gamma_1$ at two distinct points (say $u$ and $w$) of $\Gamma_2$, we obtain $\Gamma$. By induction, the lemma holds for $\Gamma_2$. We first claim that $\Gamma_1$ can have at most two non $\mathcal{A}$-points of $\Gamma$. Note that $u$ and $w$ are $\mathcal{A}$-points of $\Gamma$ and all vertices of $\Gamma_1$ have degree $2$ in $\Gamma$ except the endpoints. Thus, if $C$ is a cycle in $\Gamma$, then either no non $\mathcal{A}$-point of $\Gamma$ is common in $\Gamma_1$ and $C$, or $\Gamma_1$ is a subgraph of $C$. Hence, if $\Gamma_1$ has two non $\mathcal{A}$-points, then all cycles intersecting $\Gamma_1$ have at least two non $\mathcal{A}$-points. We note that $\Gamma_2$ is free from bad cycles, and all simple cycles that do not share an edge with $\Gamma_1$ lie in $\Gamma_2$.
Further, simple cycles which have a common edge with $\Gamma_1$ are not bad cycles, provided $\Gamma_1$ has two non $\mathcal{A}$-points. Thus, the claim follows. 

Now, the following two possibilities arise.

     \begin{enumerate}
     \item \textit{Case $1$: $\Gamma_1$ has at least one non $\mathcal{A}$-point.}

 In this case, to construct the net $\Gamma$ iteratively, we can first construct $\Gamma_2$ iteratively, and then glue the segment $\Gamma_1$ having either one or two non $\mathcal{A}$-points.

     \item \textit{Case $2$: $\Gamma_1$ has no non $\mathcal{A}$-point.}

Since $u, w \in \Gamma_2$ and $\Gamma_2$ is connected, there exists a simple path $P_1 : u \to u_1 \to u_2 \to \cdots \to u_s \to w$ in $\Gamma_2$ between $u$ and $w$. The path $P_1$ along with $\Gamma_1$ forms a simple cycle, say $C$ in $\Gamma$. Since $C$ is not a bad cycle and $\Gamma_1$ contains no non $\mathcal{A}$-point of $\Gamma$, the path $P_1$ has at least two non $\mathcal{A}$-points of $\Gamma$. We pick a point, say $u_i$, on the path $P_1$ which is a non 
$\mathcal{A}$-point of $\Gamma$.
Let $\Gamma_3$ be the maximal segment in $P_1$ containing $u_i$. Let $\Gamma_4$ be the graph obtained after deleting the edges of $\Gamma_3$ from $\Gamma$. Then by Lemma \ref{deleting maximal segments from a net}, $\Gamma_4$ is a net and by Lemma \ref{eta(Gamma) equation}, $\eta(\Gamma_4) = \eta(\Gamma) - 1 = r-1$. Now, by induction, the lemma holds for $\Gamma_4$. By an argument as above, we can show that $\Gamma_3$ has at most two non $\mathcal{A}$-points. Hence, in this case, $\Gamma$ can be constructed iteratively by first constructing $\Gamma_4$ iteratively, and then gluing the segment $\Gamma_3$ which has one or two non $\mathcal{A}$-points of $\Gamma$. 
\end{enumerate} 
\end{proof} 

The following lemma provides a sign function for a net without bad cycles.

\begin{lemma}\label{net parity}
Let $\Gamma$ be a net that does not contain a bad cycle and
$\mathcal{A}$ be an anchor for $\Gamma$. Then there exists a function $\sigma : V \to \{0, 1, -1\}$ such that 
    
\begin{enumerate}
\item $\sigma(v) = 0$ if and only if $v$ is a non $\mathcal{A}$-point of $\Gamma$.
\item If $u$, $v$ are adjacent $\mathcal{A}$-points of $\Gamma$, then $\sigma(u) = \sigma(v)$.
\item If $u$, $v$ are distinct $\mathcal{A}$-points of $\Gamma$ and there exists a non $\mathcal{A}$-point of $\Gamma$ that is adjacent to both $u$ and $v$, then $\sigma(u) = -\sigma(v)$.      
\end{enumerate}
\end{lemma}

\begin{proof}
We apply induction on $\eta(\Gamma)$. If $\eta(\Gamma)=1$, then $\Gamma$ is a cycle. If $\Gamma = C_3$, then $\mathcal{A} = \emptyset$ and hence we define $\sigma : V \to \{0, 1, -1\}$ as $\sigma(v_i) = 0$, for each $v_i \in V$. If $\Gamma = C_n$ with $n >3$; say $C_n:v_1\to v_2\to \cdots \to v_n \to v_1$, then $\Gamma$ has two non $\mathcal{A}$-points, and they cannot be adjacent. Without loss of generality, let $k>2$ be such that $v_1$ and $v_k$ are the only two non $\mathcal{A}$-points of $C_n$. We assign 
\begin{align*}
&\sigma(v_1) = \sigma(v_k) = 0, \\
&\sigma(v_2) = \cdots = \sigma(v_{k-1}) = 1 \text{ and } \\
&\sigma(v_{k+1}) = \cdots = \sigma(v_{n}) = -1
\end{align*}
to exhibit a function $\sigma$ as asserted in the lemma.

Now, assume that $r := \eta(\Gamma) \geq 2$. Suppose, the lemma holds for all nets $\Gamma'$ with $\eta(\Gamma')<r$. Let $\Gamma$ be constructed by gluing the endpoints of the segment $\Gamma_1 : u_1\to u_2\to \cdots \to u_s$ at two points of a net $\Gamma_2$ with $\eta(\Gamma_2) = r-1$. By Lemma \ref{inductive construction net}, it can be assumed that the segment $\Gamma_1$ has either one or two vertices that are not $\mathcal{A}$-points. Let $u_1$ and $u_s$ be the vertices of $\Gamma_1$ to be glued to the $\mathcal{A}$-points $w_1$ and $w_s$ of $\Gamma_2$ to obtain $\Gamma$. We have the following two cases.

\begin{enumerate}
\item \textit{Case $1$: $\Gamma_1$ has only one non $\mathcal{A}$-point of $\Gamma$.}
Let $u_t$ be the unique non $\mathcal{A}$-point of $\Gamma$ that is contained in $\Gamma_1$. There must be a cycle, say $C$ in $\Gamma$, which has only two non $\mathcal{A}$-points and $u_t$ is one of them. Let $\Gamma_3$ be the part of $C$ present in $\Gamma_2$. So, $\Gamma_3$ contains a single non $\mathcal{A}$-point of $\Gamma$ and the vertices $w_1$ and $w_s$. Thus, by induction applied to $\Gamma_2$, we conclude that $\sigma(w_1) = -\sigma(w_s)$. We define 
\begin{align*}
&\sigma(u_1) = \cdots = \sigma(u_{t-1}) = \sigma(w_1), \\
&\sigma(u_{t+1}) = \cdots = \sigma(u_{s}) = \sigma(w_s) \text{ and } \\
&\sigma(u_t) = 0.
\end{align*} 
    The function $\sigma$ thus constructed is asserted in the lemma.

        \item \textit{Case $2$: $\Gamma_1$ contains two non $\mathcal{A}$-points of $\Gamma$.}

        Let $i<j$ and $u_i, u_j$ be two non $\mathcal{A}$-points in $\Gamma_1$. Thus, there must be a cycle $C$ in $\Gamma$ with exactly two non $\mathcal{A}$-points $u_i$ and $u_j$. Let $\Gamma_3$ be the part of $C$ present in $\Gamma_2$. Then $\Gamma_3$ is a segment that contains no non $\mathcal{A}$-points of $\Gamma$ and has $w_1$ and $w_s$ as endpoints. Thus, by induction on $\Gamma_2$, we conclude that $\sigma(w_1) = \sigma(w_s)$. We define
           \begin{align*}
    &\sigma(u_1) = \cdots = \sigma(u_{i-1}) = \sigma(u_{j+1}) = \sigma(u_{j+2}) \cdots = \sigma(u_{s}) = \sigma(w_1), \\
    &\sigma(u_{i+1}) = \cdots = \sigma(u_{j-1}) = -\sigma(w_1) \text{ and }\\
    &\sigma(u_i) = \sigma(u_j) = 0.
\end{align*} 
With this case, the proof of the lemma is complete.
    \end{enumerate}
\end{proof}

Let $\sigma$ be a function as in Lemma \ref{net parity}. Then $\sigma$ is called a \emph{sign} function on the graph $\Gamma$. A vertex $v$ of $\mathcal{A}$ is said to be of \textit{positive $\sigma$-parity} if $\sigma(v) = +1$, and \textit{negative $\sigma$-parity} if $\sigma(v) = -1$. We define $\sigma^{-}$ by $\sigma^{-}(v) = -\sigma(v)$ for all $v \in V$, and note that $\sigma^{-}$ is also a sign function. Thus, for a given $v \in V$, the Lemma \ref{net parity} asserts the existence of a sign function $\sigma$ such that the parity of $v$ with respect to $\sigma$ is non-negative.

\section{System of Balance Equations and Labeling of Graphs}\label{Labeling of Graphs}
We define a system of balance equations $E(D)$ over a commutative ring $R$ and associate an $R$-weighted graph $\Gamma(D)$ to it.
\subsection{Balance Equations}\label{Section on balance equations} Let $R$ be a commutative ring and $\varepsilon$ be a symbol. For an integer $n > 1$, let $$A(n) := \{(i,j) : 1 \leq i < j \leq n\},$$ 
and $D : A(n) \to R \cup \{\varepsilon\}$ be a function. Denote $d_{i,j} := D(i,j)$.
By \emph{support} of $D$ we mean the set 
$${\rm supp}(D) := \{(i,j) \in A(n) : d_{i,j} \neq \varepsilon\}.$$
Let $\mu(D)$ denote the number of elements in ${\rm supp}(D)$
and $m(D)$ denote the number of integers $k$ such that for some $1 \leq i \leq n$, either $(i,k)$ or $(k,i)$ lies in ${\rm supp}(D)$. 
We lose nothing but simplify notation by assuming $n = m(D)$, and hence take $n = m(D)$, throughout.

For $(i,j) \in {\rm supp}(D)$, we formulate the equation
$$E_{i,j}(D): x_i y_j - x_j y_i = d_{i,j}.$$
Let $E(D)$ denote the system of these equations as $(i,j)$ vary over ${\rm supp}(D)$. The system $E(D)$ consists of $\mu(D)$ equations in $2n$ variables. Each equation $E_{i,j}(D)$ in this system is called a \emph{balance equation}. We look for solutions of $E(D)$ in $R$.
Let $(\alpha, \beta) \in R^n \times R^n$, be a solution for $E(D)$, and $\alpha_k, \beta_k$ denote $k^{\rm th}$ coordinates of $\alpha$ and 
$\beta$, respectively. Then $(\alpha_k, \beta_k) \in R \times  R$ is called the 
$k^{\rm th}$-\emph{solution pair} of $(\alpha,\beta)$.

The system $E(D)$ can be represented as an $R$-weighted graph $\Gamma(D)$ with $n$ vertices $\{v_1, v_2,\cdots, v_n\}$, where for each $(i,j) \in {\rm supp}(D)$ there is an edge between $v_i$ and $v_j$ with weight $d_{i,j}$.
By a \emph{labeling} of vertices in $\Gamma(D)$, we mean an assignment $v_k \mapsto (a_k, b_k) \in R \times R$, for each vertex $v_k$. We call $a_k$ the \emph{$x$-label} and $b_k$ the \emph{$y$-label} of this labeling. A labeling of vertices in $\Gamma(D)$ is called \emph{consistent} if
$a_ib_j - a_jb_i = d_{i,j}$ for every $(i, j) \in {\rm supp}(D)$. Thus, a consistent labeling of vertices in $\Gamma(D)$ corresponds to a solution of the system $E(D)$, and vice versa. 

\subsection{Labeling of R-weighted graphs}\label{Labeling of R-weighted graphs}
This subsection is devoted to solving the system $E(D)$, or equivalently, to provide a consistent labeling on $\Gamma(D)$. Since variables $x_i, y_i$ across disjoint components of a graph do not interact in $E(D)$, it is safe to assume that $\Gamma(D)$ is a connected graph. We start with the case when $\Gamma(D)$ is a tree.

\begin{lemma}\label{labeling-a-tree}
If $\Gamma(D)$ is a tree, then it admits a consistent labeling for any ring $R$.
\end{lemma}

\begin{proof}
Let $n = m(D)$ and $\alpha = (\alpha_1, \alpha_2, \dots, \alpha_n) \in R^n$ be arbitrary, except for the restriction that each $\alpha_i$ is invertible in $R$. We argue by induction on $n$. If $n = 2$, then
we assign $\beta_1 = b$, where $b$ is an arbitrary element in $R$, and $\beta_2 = (b\alpha_2 + d_{1,2}) {\alpha_1}^{-1} \in R$. Then $v_i \mapsto (\alpha_i, \beta_i)$ is indeed a consistent labeling.

Now, suppose $n > 2$. Let $1 \leq k, \ell \leq n$ be such that $v_\ell$ is a pendant vertex adjacent to $v_k$. A re-enumeration of vertices of $\Gamma(D)$ allows us to assume that $k=n-1$ and $l=n$. The induction hypothesis ascertains a consistent labeling $v_i \mapsto (\alpha_i, \beta_i)$ on the tree obtained by removing $v_n$ from $\Gamma(D)$. Now, we put $\beta_{n} = ({\beta}_{n-1} \alpha_{n} + d_{n-1,n}) {\alpha_{n-1}}^{-1} \in R$ to extend the labeling to $\Gamma(D)$ and make $v_i \mapsto (\alpha_i, \beta_i)$ a consistent labeling on $\Gamma(D)$.
\end{proof}

\begin{remark}\label{freedom-in-labeling-a-tree}
It follows from the proof of Lemma \ref{labeling-a-tree} that if $\Gamma(D)$ is a tree, then not only does a consistent labeling exist, but there is also considerable freedom in finding one. In fact, any arbitrary $\alpha \in R^n$ with invertible coordinates can be used for $x$-labels, and the $y$-label of one of the vertices can be arbitrarily assigned to construct a consistent labeling. Reciprocally, a consistent labeling can be constructed by assigning an arbitrary $\beta \in R^n$, with invertible coordinates, as $y$-labels, and the $x$-label of one of the vertices can be arbitrarily assigned.
\end{remark}

Let $R$ be a local ring with maximal ideal $\mathfrak m$. The residue class of $x \in R$ modulo $\mathfrak m$ will be denoted by $[x]$. Thus, $x \in R$ is non-invertible if and only if $[x] = [0]$. We denote by $R^{\times}$ the set of invertible elements in $R$. Thus, $R^{\times} = R \setminus \mathfrak m$.

We now define an unfavorable proximity for a bad cycle in a graph $\Gamma$.

Let $C_r: u_{1} \to u_{2} \to \dots \to u_{r} \to u_1$ be a bad cycle in a graph $\Gamma$ such that $\deg_{\Gamma}(u_{i}) > 2$ for $i \leq r-2$, and $\deg_{\Gamma}(u_{r-1})= \deg_{\Gamma}(u_{r}) = 2$. For each vertex $u_{i}$, $ 1 \leq i \leq r-2$ in $C_r$, let us pick an adjacent vertex $u_{i+r}$ outside $C_r$. Let $\mathcal{P}$ be the subgraph of $\Gamma$ containing $C_r$, vertices $u_{i+r}$, and the edges between $u_i$ and $u_{i+r}$. The subgraph $\mathcal{P}$ is said to be a \textit{proximity} for $C_r$ in $\Gamma$.

\begin{definition}\label{unfavorable proximity}
Let $\mathcal{P}$ be a proximity in the vertices $v_{i_1}, v_{i_2}, \cdots v_{i_{2r-2}}$ in a weighted graph $\Gamma(D)$ such that $C_r: v_{i_1} \to v_{i_2} \to \dots \to v_{i_r} \to v_{i_1}$ is a bad cycle and the vertex $v_{i_j}$ is adjacent to the vertex $v_{i_{j+r}}$ for each $j \in \{1,2,\cdots, r-2\}$. Then $\mathcal{P}$ is said to be \textit{unfavorable} if $[d_{i_1, i_2}]= [d_{i_2, i_3}] = \dots [d_{i_{r-2}, i_{r-1}}] = [d_{i_1, i_r}] = [0]$ and $[d_{i_{r-1}, i_r}]$, $[d_{i_1, i_{r+1}}]$, $[d_{i_2, i_{r+2}}], \dots , [d_{i_{r-2}, i_{2r-2}}] \ne [0]$.
\end{definition}

The following theorem shows that a consistent labeling is impossible when an unfavorable proximity is present in the graph.
\begin{theorem}\label{consistent labeling doesn't exist}
Let $R$ be a local ring and $D: A \to R \cup \{\epsilon\}$ be a function such that the graph $\Gamma(D)$ contains a bad cycle with unfavorable proximity. Then $\Gamma(D)$ does not admit a consistent labeling.
\end{theorem}

\begin{proof}
Note that it is enough to show that the theorem holds when $\Gamma(D)$ is itself an unfavorable proximity for some simple cycle $C_r$. Let the vertex set of $\Gamma(D)$ be $V:=\{ v_1, v_2, \dots v_{2r-2}\}$ such that the vertices $v_1, v_2. \dots, v_r$ form the cycle $C_r$ and the vertex $v_{r+i}$ is adjacent to the vertex $v_i$ in $\Gamma(D)$ for $i \in \{ 1, 2, \dots, r-2\}$. Since $\Gamma(D)$ is an unfavourable proximity of $C_r$, we have
$[d_{1,2}] = [d_{2,3}] = \dots = [d_{r-2,r-1}] = [d_{1,r}] = [0]$ and $[d_{r-1,r}] \ne [0], [d_{1,r+1}] \ne [0], [d_{2,r+2}] \ne [0], \dots , [d_{r-2,2r-2}] \ne [0]$.

Assume to the contrary that $v_k \mapsto (\alpha_k, \beta_k)$ is a consistent labeling of $\Gamma(D)$. Observe that every vertex $v_i$ of $\Gamma(D)$ supports an edge carrying a nonzero weight modulo $\mathfrak m$. Thus, for each $i \in \{ 1, 2, \dots , 2r-2\}$, both $[\alpha_i]$ and $[\beta_i]$ can not be simultaneously $[0]$.

We claim that if $i \in \{ 1, 2, \dots , r\}$, then $[\alpha_i]  \ne [0]$ and $[\beta_i] \ne [0]$. To show this, let $j \in \{1, 2, \cdots , r\}$ be such that exactly one of $[\alpha_j]$ or $[\beta_j]$ is equal to $[0]$. Without loss of generality, let $[\alpha_j] = [0]$ and $[\beta_j] \ne [0]$. We first show that $[\alpha_i] = [0]$ and $[\beta_i] \ne [0]$ for each $i \in \{ 1,2, \cdots , r\}$. For this purpose, We consider the following three cases depending on the $j$.

\noindent\textbf{Case 1}: $j=r$. \quad Since $(\alpha_1, \beta_1)$ and $(\alpha_r, \beta_r)$ are solution pairs of $v_1$ and $v_r$, respectively, we have $\alpha_1 \beta_r - \alpha_r \beta_1 = d_{1,r}$. Further, since $[d_{1,r}] = [0]$, $[\alpha_r] = [0]$ and $[\beta_r] \ne [0]$, we have $[\alpha_1] = [0]$. Thus, $[\beta_1] \ne [0]$, as $[\alpha_1]$ and $[\beta_1]$ are not simultaneously $[0]$. Now, since $(\alpha_1, \beta_1)$ and $(\alpha_2, \beta_2)$ are solution pairs of $v_1$ and $v_2$, respectively,
we have $\alpha_1 \beta_2 - \alpha_2 \beta_1 = d_{1,2}$. Further, since $[d_{1,2}] = [0]$, $[\alpha_1] = [0]$ and $[\beta_1] \ne [0]$, we have $[\alpha_2] = [0]$, and thus $[\beta_2] \ne [0]$. Similarly, for each $i \in \{1,2,\cdots, r-2\}$, we deduce inductively that $[\alpha_{i+1}] = [0]$ and $[\beta_{i+1}] \ne [0]$ using the conditions on $d_{i, i+1}$, $\alpha_i$ and $\beta_i$.

\noindent \textbf{Case 2:} $j=r-1$. \quad In this case, inductively for each $i \in \{ 2, 3, \cdots , r-1\}$, we conclude $[\alpha_{i-1}] = [0]$ and $[\beta_{i-1}] \ne [0]$, using the conditions on $d_{i-1, i}$, $\alpha_i$ and $\beta_i$. Finally, since $\alpha_1 \beta_r - \alpha_r\beta_1 = d_{1,r}$, we conclude that $[\alpha_r] = [0]$ and $[\beta_r] \ne [0]$.

\noindent \textbf{Case 3:} $j \in \{ 1, 2, \cdots , r-2\}$. \quad In this case, if $i \in \{ j, j+1, \cdots , r-2 \} $, then inductively we get $[\alpha_{i+1}] = [0]$ and $[\beta_{i+1}] \ne [0]$, using the conditions on $d_{i, i+1}$, $\alpha_i$ and $\beta_i$; and if $i \in \{ j, j-1, \cdots , 2 \} $, then inductively we get $[\alpha_{i-1}] = [0]$ and $[\beta_{i-1}] \ne [0]$, using the conditions on $d_{i-1, i}$, $\alpha_i$ and $\beta_i$. 

Now, since $[\alpha_{r-1}] = [0]$ and $[\alpha_r] = [0]$, from the relation $\alpha_{r-1} \beta_r - \alpha_r \beta_{r-1} = d_{r-1,r}$, it follows that $[d_{r-1,r}] = [0]$. This is a contradiction to the assumption that $[d_{r-1,r}] \ne [0]$. Thus, the claim holds. This facilitates further computation involving inverses of $\alpha_i$ and $\beta_i$ when $i \in \{1,2,\cdots,r\}$.

Since $\alpha_1$ is invertible in $R$, and
\begin{align} \label{E(1,r)}
\alpha_1 \beta_r - \alpha_r \beta_1 = d_{1,r}
\end{align}
\begin{align}\label{E(r-1,r)}
\alpha_{r-1}\beta_{r} - \alpha_{r} \beta_{r-1} = d_{r-1, r}
\end{align}
extracting $\beta_r$ from \ref{E(1,r)} and substituting it in \ref{E(r-1,r)}, we get
\begin{align}\label{d(r-1,r)}
d_{r-1, r} = \alpha_{r-1}(\alpha_r \beta_1 + d_{1,r}){\alpha_1}^{-1} - \alpha_{r} \beta_{r-1} = (\alpha_{r-1}\alpha_r ( \beta_1 {\alpha_1}^{-1}) - \alpha_{r} \beta_{r-1} )+\alpha_{r-1} d_{1,r} {\alpha_1}^{-1}.
\end{align}
Similarly, we have
\begin{align} \label{E(1,2)}
\alpha_1 \beta_2 - \alpha_2 \beta_1 = d_{1,2}. 
\end{align}

We use the invertibility of $\alpha_2$ to extract $\beta_1$ from \ref{E(1,2)}, and substitute it in \ref{d(r-1,r)} to obtain 
\begin{align}\label{d(r-1,r)2}
 d_{r-1, r} = (\alpha_{r-1}\alpha_r (\beta_2 {\alpha_2}^{-1}) - \alpha_{r} \beta_{r-1}) - \alpha_{r-1}\alpha_r d_{1,2}{\alpha_1}^{-1} {\alpha_2}^{-1} +\alpha_{r-1} d_{1,r} {\alpha_1}^{-1}.
\end{align}

We proceed along the same pattern and use invertibility of $\alpha_{i+1}$
in each of the equations $\alpha_i \beta_{i+1} - \alpha_{i+1}\beta_i = d_{i,i+1}$, for $i \in \{2,3,\cdots,r-2\}$, to finally obtain
\begin{align}\label{d(r-1,r)3}
d_{r-1, r} &= (\alpha_{r-1}\alpha_r (\beta_{r-1} {\alpha_{r-1}}^{-1}) - \alpha_r \beta_{r-1}) - \alpha_{r-1}\alpha_r\left(\sum_{i=1}^{r-2} d_{i,i+1} \alpha_{i}^{-1}\alpha_{i+1}^{-1}\right) + \alpha_{r-1} d_{1,r} {\alpha_1}^{-1} \\
&= - \alpha_{r-1}\alpha_r\left(\sum_{i=1}^{r-2} d_{i,i+1} \alpha_{i}^{-1}\alpha_{i+1}^{-1}\right) + \alpha_{r-1} d_{1,r} {\alpha_1}^{-1}. \notag
\end{align}
Since $[d_{1,r}] = [d_{1,2}] = [d_{2,3}] = \cdots [d_{r-2,r-1}] = [0]$, from \ref{d(r-1,r)3}
we obtain $[d_{r-1, r}] = [0]$, which is a contradiction. Hence, the assumption that $\Gamma(D)$ admits a consistent labeling is wrong, and the theorem follows.
\end{proof}



To obtain more specific results on labelings of $\Gamma(D)$, we consider $R$ to be a local ring with at least three residue classes modulo its maximal ideal $\mathfrak m$. The following lemma guarantees a consistent labeling when $\Gamma(D)$ is a cycle in more than $3$ vertices.

\begin{lemma}\label{labeling-a-cycle}
Let $R$ be a local ring with at least three residue classes modulo its maximal ideal $\mathfrak m$. For an integer $n > 3$, let $D : A(n) \to R  \cup \{\varepsilon\}$ be a function with 
${\rm supp}(D) = \{(i,i+1) : 1 \leq i \leq n-1\} \cup \{(1,n)\}$ so that $\Gamma(D)$ is a cycle. Let $v_r \neq v_n$ be a vertex in $\Gamma(D)$
that is not adjacent to $v_n$. Let $s < r$. Then for each $a, c \in R^{\times}$ and $b \in R$, the graph $\Gamma(D)$ admits a consistent labeling $v_k \mapsto (\alpha_k, \beta_k)$ such that 
\begin{enumerate}
\item[(i).] $\alpha_s = a$, $\beta_s = b$, $\beta_{r+1} = c$.
\item[(ii).] $\alpha_k \in R^{\times}$, whenever $k < r$.
\item[(iii).] $\beta_k \in R^{\times}$, whenever $r < k  < n$.
\end{enumerate}


\end{lemma}

\begin{proof}
For labeling vertices of $\Gamma(D)$, we split the vertex set into four parts and deal with them one by one.
The indices of these sets are: 
\begin{enumerate}
\item[(a).] $S_1 := \{k \in \mathbb N : s \leq \ell \leq r-1\}$
\item[(b).] $S_2 := \{k \in \mathbb N : r \leq \ell \leq n-1\}$
\item[(c).] $S_3 := \{k \in \mathbb N : 1 \leq \ell \leq s-1\}$
\item[(c).] $S_4 := \{n\}$
\end{enumerate}

\noindent{\bfseries Step 1}. Labeling the set $\{v_k : k \in S_1\}$, where $S_1:= \{k\in \mathbb N : s \leq k \leq r-1\}$. 

First, we set $\alpha_s = a$, $\beta_s = b$. 
We claim that for each $k \in S_1$, there exist $\alpha_{k} \in R^{\times}$ such that for the iteratively defined sequence $\{\beta_k\}_{k \in S_1}$, with $\beta_s = b$ and
$\beta_{k} = (d_{k-1,k} + \alpha_{k}\beta_{k-1})\alpha_{k-1}^{-1}$ for $s < k \in S_1$, we have
either $[\beta_{k}] = [0]$, or 
$$ [\beta_{k} {\alpha_{k}}^{-1}] = \begin{cases}
  [b a^{-1}], ~~~~~~~\quad \text{ if }[b] \ne [0], \\
  [a^{-1}], \quad \quad \text{ if }[b] = [0].
\end{cases} $$

We proceed to prove this claim. Let $k_1 < \cdots < k_t$ be all indices with $k_0 := s+1 \leq  k_{\ell} \leq r-1$
such that $[d_{k_\ell-1, k_{\ell}}] \neq [0]$. Thus, $[d_{k-1,k}] = [0]$, whenever $k_{\ell-1} < k < k_{\ell}$ for some $\ell \in \{1,2, \cdots, t\}$ or $k_t < k \leq r-1$. 

Note that since $\alpha_{i-1}\beta_i - \alpha_i\beta_{i-1} = d_{i-i,i}$, if $[d_{i-1,i}] =[0]$ and $[\alpha_{i-1}], [\alpha_{i}] \neq [0]$, then $[\beta_{i}] \neq [0]$ if and only if $[\beta_{i-1}] \neq [0]$. Moreover, $[\beta_{i} \alpha_{i}^{-1}] = [\beta_{i-1} \alpha_{i-1}^{-1}]$. Thus, for $s < i < k_1$ we have $[\beta_{i} \alpha_{i}^{-1}] = [a^{-1}b]$. Similarly, if $k_\ell < i < k_{\ell+1}$ for some $\ell \geq 1$, then $[\beta_{i} \alpha_{i}^{-1}] = [\beta_{k_\ell} \alpha_{k_\ell}^{-1}]$. Therefore, to examine all possibilities of $[\beta_i \alpha_{i}^{-1}]$, as $i \in S_1$, it is enough to assume that $i \in \{k_0, k_1, k_2, \cdots , k_t\}$. Consider the two cases, depending on the residue class of $[b]$.

Suppose $[b] = [0]$. If $s+1 < k_1$, then $[\beta_{s+1}] = 
[(d_{s,s+1} + \alpha_{s+1}\beta_s)a^{-1}]=
[\alpha_{s+1}\beta_s a^{-1}] = [\alpha_{s+1} b a^{-1}] = [0]$, for any choice $\alpha_{s+1} \in R^{\times}$. However, if $s+1 = k_1$, then $d_{s,s+1} \in R^{\times}$ and hence $[\beta_{s+1}] = [d_{s,s+1} a^{-1}] \ne [0]$. Thus, $\beta_{s+1}$ is independent of $\alpha_{s+1}$. We choose $\alpha_{s+1} = \beta_{s+1}a \in R^{\times}$ so that $[\beta_{s+1} \alpha_{s+1}^{-1}] = [a^{-1}]$.
Now, let $k > s+1$ and $k \in \{k_1, k_2, \cdots, k_t\}$. If $[\beta_{k-1}]=[0]$, then $[\beta_{k}] = [d_{k-1,k}\alpha_{k-1}^{-1}] \neq [0]$, which is independent of the choice of $\alpha_k$. We choose $\alpha_k = \beta_k a \in R^{\times}$, so that $[\beta_{k} \alpha_{k}^{-1}] = [a^{-1}]$. Further, if $[\beta_{k-1}] \neq [0]$, then we choose $\alpha_k = -\beta_{k-1}^{-1} d_{k-1,k} \in R^{\times}$, so that $\beta_k = 0$. Thus, the claim holds when $[b] = [0]$.

Now, suppose $[b] \neq [0]$. If $s+1 < k_1$, then $[d_{s,s+1}] = [0]$. Thus,
$[\beta_{s+1}]=[(d_{s,s+1} + \alpha_{s+1} b)a^{-1}] = [\alpha_{s+1} b a^{-1}]$, and hence
$[\beta_{s+1} \alpha_{s+1}^{-1}] = [b a^{-1}]$. If $s+1 = k_1$, then $d_{s,s+1} \in R^{\times}$, in which case we choose $\alpha_{s+1} = -{b}^{-1}d_{s,s+1} \in R^\times$, and obtain $\beta_{s+1} = 0$.
For $s+1 < k \in \{k_1, k_2, \cdots, k_t\}$, if $[\beta_{k-1}]=[0]$, then $[\beta_{k}] = [d_{k-1,k}\alpha_{k-1}^{-1}] \neq [0]$, making 
$[\beta_k]$ independent of $\alpha_k$. We choose $\alpha_{k} = \beta_{k}b^{-1}a$. Clearly, $[\beta_{k} \alpha_{k}^{-1}] = [b a^{-1}]$. Further, if $[\beta_{k-1}] \ne [0]$, then we choose $\alpha_{k} = -\beta_{k-1}^{-1} d_{k-1,k} \in R^{\times}$, so that $\beta_{k} = 0$. This shows the claim in this case.

\noindent{\bfseries Step 2}. Labeling the set $\{v_k : k \in S_2\}$, where $S_2:= \{k \in \mathbb N : r \leq k \leq n-1\}$.

Set $\alpha_{r} = d_{r,r+1}c^{-1}$ and $\beta_{r} = (d_{r-1,r} + d_{r,r+1}c^{-1}\beta_{r-1})\alpha_{r-1}^{-1}$, where $\alpha_{r-1}$ and $\beta_{r-1}$ are as in the previous case. Further, choose $\alpha_{r+1} =0$ and $\beta_{r+1}=c$.
Fix a $\gamma \in R^{\times}$ with the following property: $\gamma - a^{-1}b \in R^{\times}$, if $[b] \neq [0]$ and $\gamma - a^{-1} \in R^{\times}$, if $[b] = [0]$. Such a $\gamma$ exists because $|R/\mathfrak m| \geq 3$.

We claim that for each $k \in S_2\setminus \{r\}$, there exists $\beta_{k} \in R^{\times}$ such that for iteratively defined sequence $\{\alpha_k\}_{k \in  S_2\setminus \{r\}}$, with $\alpha_{r+1} =0$, $\beta_{r+1}=c$ and $\alpha_{k} = (\alpha_{k-1} \beta_{k} - d_{k-1,k})\beta_{k-1}^{-1}$ for $r+1 < k \in S_2$, we have either $[\alpha_k] = [0]$, or $[\beta_k {\alpha_k}^{-1}] = [\gamma]$.

Let $\ell \geq r+2$ be the smallest index such that $d_{\ell -1, \ell} \in R^{\times}$. Thus, for $r+2 \leq k < \ell$, we have $[d_{k-1,k}] = [0]$, and consequently
$$[\alpha_{k}] = [(\alpha_{k-1} \beta_{k} - d_{k-1,k})\beta_{k-1}^{-1}] = [\alpha_{k-1} \beta_{k}\beta_{k-1}^{-1}] = [\alpha_{k-2} \beta_{k}\beta_{k-2}^{-1}] = \cdots = [\alpha_{r+1} \beta_{k}\beta_{r+1}^{-1}] = [0],$$
for any choice of $\beta_k \in R^{\times}$. Since $\alpha_{\ell} = (\alpha_{\ell-1} \beta_{\ell} - d_{\ell-1,\ell})\beta_{\ell-1}^{-1}$, we have $[\alpha_{\ell}] = [ - d_{\ell-1,\ell}\beta_{\ell-1}^{-1}] \ne [0]$. Note that $[\alpha_{\ell}]$ is independent of $[\beta_{\ell}]$. We choose $\beta_{\ell} = \gamma \alpha_\ell \in  R^{\times}$. Hence, $[\alpha_{\ell}^{-1}\beta_{\ell}] = [\gamma]$. Now, for labeling each $v_k$ with $\ell < k \leq n-1$, we imitate what we did in Step 1.

\noindent{\bfseries Step 3}. Labeling the set $\{v_k : k \in S_3\}$, where $S_3:= \{k \in \mathbb N : 1 \leq k \leq s-1\}$.

Recall that we have assigned $\alpha_s = a$, $\beta_s = b$ in Step 1. Now, inductively define $\beta_{k} = (\alpha_{k} \beta_{k+1} - d_{k,k+1})\alpha_{k+1}^{-1}$, where the choices of $\alpha_{k} \in R^{\times}$ are made in a manner similar to Step 1 so that either $[\beta_{k}]=[0]$, or 
$[\alpha_{k}^{-1}\beta_{k}] = [a^{-1}]$ if $[b] = [0]$ and $[\alpha_{k}^{-1}\beta_{k}] = [a^{-1}b]$ if $[b] \ne [0]$.

\noindent{\bfseries Step 4}. Labeling the set $S_3:= \{n\}$.

Finally, we label $v_n$. Note that we have already labeled $v_{n-1}$ in Step 2 and $v_1$ just now in Step 3. Let $\beta_{n} = (\alpha_{n} \beta_{1} + d_{1,n})\alpha_{1}^{-1}$. Substitute it into the equation $\alpha_{n-1} \beta_n -\alpha_n \beta_{n-1} = d_{n-1,n}$. Then $\alpha_n (\alpha_{n-1}\beta_1 \alpha_1^{-1} - \beta_{n-1}) = d_{n-1,n} - \alpha_{n-1}d_{1,n}\alpha_1^{-1}$. This equation has a solution for $\alpha_n$ if $\alpha_{n-1}\beta_1 \alpha_1^{-1} - \beta_{n-1} \in R^{\times}$. Since $\alpha_1, \beta_{n-1} \in R^{\times}$, this holds if either $[\alpha_{n-1}] = [0]$ or $[\beta_1] = [0]$. However, if $[\beta_1] \ne [0]$ and $[\alpha_{n-1}] \ne [0]$, then by Step 2 and Step 3, $[\alpha_{1}^{-1}\beta_{1}] = [a^{-1}b]$ if $[b] \neq [0]$, $[\alpha_{1}^{-1}\beta_{1}] = [a^{-1}]$ if $[b] = [0]$, and $[\alpha_{n-1}^{-1}\beta_{n-1}] = [\gamma]$. Thus, by the definition of $\gamma$, we have that $\alpha_{n-1}^{-1}\beta_{n-1} - \alpha_{1}^{-1}\beta_{1} \in R^{\times}$. Therefore, $\alpha_{n-1}\beta_1 \alpha_1^{-1} - \beta_{n-1} \in R^{\times}$. This establishes the lemma. 
\end{proof}

The following corollary does not require the assumption $|R/\mathfrak m| > 2$.

\begin{corollary}\label{labeling-a-cycle-p=2}
Let $R$ be a local ring. For an integer $n > 3$, let $D : A(n) \to R  \cup \{\varepsilon\}$ be a function with 
${\rm supp}(D) = \{(i,i+1) : 1 \leq i \leq n-1\} \cup \{(1,n)\}$ so that $\Gamma(D)$ is a cycle. Let $v_r \neq v_n$ be a vertex in $\Gamma(D)$
that is not adjacent to $v_n$. Then for each $a_1, a_2, a_3, a_4 \in R^{\times}$, the graph $\Gamma(D)$ admits a consistent labeling $v_k \mapsto (\alpha_k, \beta_k)$ such that 
\begin{enumerate}
\item[(i).] $\alpha_k \in R^{\times}$, whenever $k < r$.
\item[(ii).] Either $[\beta_k] = [0]$ or $[\beta_k {\alpha_k}^{-1}] = [a_1]$, whenever $k < r$.
\item[(iii).] $\beta_k \in R^{\times}$, whenever $r < k < n$.
\item[(iv).] Either $[\alpha_k] = [0]$ or $[\beta_k {\alpha_k}^{-1}] = [a_2]$, whenever $k < r$.
\item[(v).] $\alpha_1 = a_3$ and $\beta_{r+1} = a_4$.
\end{enumerate}
\end{corollary}

\begin{proof}
\noindent{\bfseries Step 1}. Labeling the set $\{v_k : k \in S_1\}$, where $S_1:= \{k \in \mathbb N : 1 \leq k \leq r-1\}$.

We first set $\alpha_1 = a_3$, $\beta_1 = 0$ and claim that for each $k \in \{1,2,\cdots, r-1\}$, there exists $\alpha_{k} \in R^{\times}$ such that for the iteratively defined sequence $\{\beta_k\}_{k \in  \{1,2,\cdots, r-1\}}$, with $\alpha_{1} =a_3$, $\beta_{1}=0$ and $\beta_{k} = (\alpha_{k} \beta_{k-1} + d_{k-1,k})\alpha_{k-1}^{-1}$ for $1 < k \leq r$, we have either $[\beta_k] = [0]$, or $[\beta_k {\alpha_k}^{-1}] = [a_1]$. This follows by an argument similar to the Step $2$ of Lemma \ref{labeling-a-cycle}, by fixing $\gamma = a_1$, reversing the role of $\alpha_k$ and $\beta_k$, and varying the indices over the set $\{1, 2, \cdots, r-1 \}$ instead of $\{r+1, r+2, \cdots, n-1\}$. 

\noindent{\bfseries Step 2}. Labeling the set $\{v_k : k \in S_2\}$, where $S_2:= \{k \in \mathbb N : r \leq k \leq n-1\}$.

This is achieved by following an argument similar to the Step $2$ of Lemma \ref{labeling-a-cycle} by fixing $\gamma = a_2$.

Finally, since $\beta_1 = 0$ and $\beta_{n-1} \in R^{\times}$, the element $\alpha_{n-1}\beta_1\alpha_1^{-1} - \beta_{n-1} = -\beta_{n-1}\in R^{\times}$. Thus, from Step $4$ of Lemma \ref{labeling-a-cycle}, a consistent choice for $\alpha_n$ and $\beta_n$ exists, by taking $\alpha_n = (d_{n-1,n} - \alpha_{n-1}d_{1,n}\alpha_1^{-1}) \beta_{n-1}^{-1}$ and $\beta_n = d_{1,n}\alpha_1^{-1}$.
\end{proof}

When $n=3$, under the further restriction that the maximal ideal
$\mathfrak m$ of $R$ is principal, we show in the following lemma that the triangular graph admits a consistent labeling.

\begin{lemma}\label{consistent labeling for a triangle}
Let $R$ be a local ring in which the unique maximal ideal $\mathfrak {m}$ is a principal ideal. Let $D: A(3) \to R \cup \{\epsilon\}$ be a function such that the graph $\Gamma(D)$ is a triangle. Then $\Gamma(D)$ admits a consistent labeling.
\end{lemma}

\begin{proof}
We split the proof into two cases.

\noindent\textbf{Case 1}: \emph{At least one of $d_{1,2}$, $d_{2,3}$ and $d_{1,3}$ is in $R^{\times}$}. \quad Without loss of generality, let $d_{2,3} \in R^{\times}$. Then the graph $\Gamma(D)$ admits a consistent labeling $v_k \mapsto (\alpha_k, \beta_k)$, where $\alpha_1 = d_{1,3}$, $\alpha_2 = d_{2,3}$, $\alpha_3 = 0$, $\beta_1 = d_{2,3}^{-1}(d_{1,3}- d_{1,2})$ and $\beta_2 = \beta_3 = 1$.

\noindent\textbf{Case 2}: \emph{$[d_{1,2}] = [d_{2,3}] = [d_{1,3}] = [0]$}. \quad Let $a \in R$ be such that $\mathfrak {m} = (a)$. Since $d_{1,2}, d_{2,3}, d_{1,3} \in \mathfrak{m}$, there exist $\ell_{i,j} \in \mathbb{N}$ and $\gamma_{i,j} \in R^{\times}$ such that $d_{i,j} = \gamma_{i,j} a^{\ell_{i,j}}$. Without loss of generality, let $\ell_{1,2} \geq \ell_{1,3}$. Then the graph $\Gamma(D)$ admits a consistent labeling $v_k \mapsto (\alpha_k, \beta_k)$, where $\alpha_1 = d_{1,3}$, $\alpha_2 = d_{2,3}$, $\alpha_3 = 0$, $\beta_1 = 0$, $\beta_2 = \left(\gamma_{1,3}^{-1} \gamma_{1,2}\right) a^{\ell_{1,2} - \ell_{1,3}}$ and $\beta_3 = 1$. 
\end{proof}

The following lemma provides a consistent labeling on a graph obtained by gluing two graphs with consistent labelings.

\begin{lemma}
Let $R$ be a commutative ring with unity and $D_1: A(n) \to R \cup \{\epsilon\}$, $D_2 : A(m) \to R \cup \{\epsilon\}$ be two functions. Let $\{v_1^{(1)}, v_2^{(1)}, \cdots, v_n^{(1)}\}$ be the set of vertices of the graph $\Gamma(D_1)$ and $\{v_1^{(2)}, v_2^{(2)}, \cdots, v_m^{(2)}\}$ be that of $\Gamma(D_2)$. Suppose $\Gamma(D_1)$ admits a consistent labeling $v_k^{(1)} \mapsto (\alpha_k, \beta_k)$, $\Gamma(D_2)$ admits a consistent labeling $v_k^{(2)} \mapsto (\gamma_k, \delta_k)$, and that
$(\alpha_i, \beta_i) = (\gamma_j, \delta_j)$ for some $v_i^{(1)}$ and $v_j^{(2)}$.
Let $D : A(n+m-1) \to R \cup \{\epsilon\}$ be a function such that $\Gamma(D) = \Gamma(D_1) \bigwedge_{v_i^{(1)}=v_j^{(2)}} \Gamma(D_2)$. Then the graph $\Gamma(D)$ admits a consistent labeling.
\end{lemma}

\begin{proof}
For $k=1,2$, let $E(D_k)$ be the system of balance equations associated with $D_k$, containing the balance equations $E_{i,j}^{(k)} (D_k):= x_{i}^{(k)} y_{j}^{(k)} - x_{j}^{(k)} y_{i}^{(k)} = d_{i,j}^{(k)}$. Since $v_i^{(1)}=v_j^{(2)} \in \Gamma(D)$, the system of balance equations corresponding to $E(D)$, is obtained by identifying variables $x_i^{(1)} = x_j^{(2)}$ and $y_i^{(1)} = y_j^{(2)}$ in the union $E(D_1) \cup E(D_2)$. Since $E(D_1)$ and $E(D_2)$ admit solutions, with $x_i^{(1)} = \alpha_i, y_i^{(1)} = \beta_i, x_j^{(2)} = \gamma_j, y_j^{(2)} = \delta_j$ and $(\alpha_i, \beta_i) = (\gamma_j, \delta_j)$, the solution naturally extends to $E(D)$, providing $\Gamma(D)$ a consistent labeling.
\end{proof}

The following lemma provides a consistent labeling on a borderless graph that is not a triangle and does not contain bad cycles.

\begin{lemma}\label{borderless}
Let $R$ be a local ring with at least three residue classes modulo its maximal ideal $\mathfrak m$ and $D: A(n) \to R \cup \{\epsilon\}$ be a function such that the graph $\Gamma(D)$ is a borderless graph that is not a triangle and does not contain bad cycles. Let $\mathcal A$ be an anchor for $\Gamma$. Let $u \in \mathcal A$, and $a, b \in R$ be such that at least one of them is invertible in $R$. Then $\Gamma(D)$ admits a consistent labeling $v_k \mapsto (\alpha_k, \beta_k)$ such that 
\begin{enumerate}
\item[(i).] $u \mapsto (a, b)$ under this labeling.
\item[(ii).] For every $v_k \in \mathcal A$, at least one of $\alpha_k$ and $\beta_k$ is invertible in $R$.
\end{enumerate}

\end{lemma}

\begin{proof}
Since the graph $\Gamma(D)$ is borderless, by Lemma \ref{inductive-gluing-for-borderless}, it can be constructed iteratively in finitely many steps by gluing either a tree or a cycle at each step. We proceed by induction on $r := s({\Gamma(D))}$. For $r = 1$, the proof follows from Lemma \ref{labeling-a-tree} and Lemma \ref{labeling-a-cycle}.

Let $r \geq 2$ and $D_i : A_i \to R \cup \{\epsilon\}$; $i=1,2$ be two functions such that 
\begin{enumerate}
\item[(a)] $\Gamma(D_1)$ is a subgraph of $\Gamma(D)$.
\item[(b)] $s(\Gamma(D_1)) = r-1$. 
\item[(c)] $\Gamma(D_2)$ is either a tree or a cycle.
\item[(d)] Gluing $\Gamma(D_2)$ to $\Gamma(D_1)$ along appropriate vertices yields the graph $\Gamma(D)$. 
\end{enumerate}
Let $v \in \Gamma(D_1)$ and $w \in \Gamma(D_2)$ be such vertices. We first claim that $v \in \mathcal A$. Since $\Gamma (D)$ is obtained after gluing $\Gamma(D_1)$ and $\Gamma(D_2)$ along $v$ and $w$, we have
$$\deg_{\Gamma (D)}(v) = \deg_{\Gamma(D_1)}(v) + \deg_{\Gamma(D_2)}(w)$$

If $\deg_{\Gamma (D)}(v) \neq 2$, then $v \in \mathcal A$, and the claim holds. Suppose $\deg_{\Gamma (D)}(v) = 2$. This is possible only when $v$ is a pendant vertex in $\Gamma(D_1)$ and $w$ is a pendant vertex in $\Gamma(D_2)$. Thus, $\Gamma(D_2)$ is a tree. Further, during the iterative construction of $\Gamma(D_1)$, the vertex $v$ would have appeared in it upon gluing a tree, say $T$, at some stage, to a graph $\Gamma(D_0)$. Thus, $\Gamma (D)$ can be constructed in lesser than $r$ steps by gluing the tree $T\bigwedge_{v=w} \Gamma(D_2)$ to the graph $\Gamma(D_0)$ at the same stage. This leads to the contradiction $s({\Gamma (D)}) < r$, and establishes our claim that $v \in \mathcal A$. 

Let us proceed with constructing consistent labeling. We now have two possibilities, $u \in \Gamma(D_1)$ or $u \in \Gamma(D_2)$. Let us first assume that $u \in \Gamma(D_1)$. 
Then $\mathcal A_1 := \mathcal A \cap V(\Gamma(D_1))$ is an anchor of $\Gamma(D_1)$ containing both $u$ and $v$.
By induction on the subgraph $\Gamma(D_1)$ and the anchor $\mathcal A_1$, we conclude that $\Gamma(D_1)$ admits a consistent labeling $v_k \mapsto (\alpha_k, \beta_k)$ with $u \mapsto (a,b)$ and at least one of $\alpha_k$ and $\beta_k$ is invertible for rest of the vertices $v_k \in \mathcal A_1$.

Now, we label $w$ in $V(\Gamma(D_2))$ with $(\alpha_{\ell}, \beta_{\ell})$, where $\ell$ is such that $v = v_{\ell}$. Applying induction on the subgraph $\Gamma(D_2)$ we obtain a consistent labeling $v_k \mapsto (\alpha_k, \beta_k)$ such that least one of $\alpha_k$ and $\beta_k$ is invertible for each vertices $v_k \in \mathcal A$. This completes labeling vertices in $\Gamma(D)$ and proves the lemma when $u \in \Gamma(D_1)$. The case when $u \in \Gamma(D_2)$ follows similarly by first applying induction on $\Gamma(D_2)$ and then on $\Gamma(D_1)$.
\end{proof}

The following lemma provides a consistent labeling on a net that is not a triangle and does not contain bad cycles.

\begin{lemma}\label{net}
Let $R$ be a local ring with at least three residue classes modulo its maximal ideal $\mathfrak m$ and $D: A \to R \cup \{\epsilon\}$ be a function such that the graph $\Gamma(D)$ is a net that is not a triangle and does not contain bad cycles. Let $\mathcal{A}$ be an anchor for $\Gamma(D)$ and $\sigma$ be a sign function on $\Gamma(D)$. Let $v_s \in \mathcal{A}$ be a vertex of positive $\sigma$-parity. Let $a \in R^{\times}$, and $b \in R$. Let $\gamma \in R^{\times}$ be such that $\gamma - a^{-1}b \in R^{\times}$, if $[b] \ne 0$ and $\gamma - a^{-1} \in R^{\times}$, if $[b] = 0$. Then $\Gamma(D)$ admits a consistent labeling $v_k \mapsto (\alpha_k, \beta_k)$ such that  

\begin{enumerate}
\item[(i).]  $[\alpha_k] \neq [0]$, if $v_k$ is a vertex of $\Gamma(D)$ with positive $\sigma$-parity.

\item[(ii).]  $[\beta_k] \neq [0]$, if $v_k$ is a vertex of $\Gamma(D)$ with negative $\sigma$-parity.

\item[(iii).] $\alpha_s = a$ and $\beta_s = b$.

\item[(iv).] If $[b] =[0]$, then
\begin{enumerate}
\item Either $[\beta_k] = [0]$ or $[\beta_k \alpha_k^{-1}] = [\alpha_s^{-1}]$, for every vertex $v_k$ of positive $\sigma$-parity.
\item Either $[\alpha_k] = [0]$ or $[\beta_k \alpha_k^{-1}] = [\gamma]$, for every vertex $v_k$ of negative $\sigma$-parity.
\end{enumerate}

\item[(v).] If $[b] \ne [0]$, then
\begin{enumerate}
\item Either $[\beta_k] = [0]$ or $[\beta_k \alpha_k^{-1}] = [\alpha_s^{-1}\beta_s]$, for every vertex $v_k$ of positive $\sigma$-parity.
\item Either $[\alpha_k] = [0]$ or $[\beta_k \alpha_k^{-1}] = [\gamma]$, for every vertex $v_k$ of negative $\sigma$-parity.
\end{enumerate}
\end{enumerate}
\end{lemma}

\begin{proof}
The proof is by induction on $\eta(\Gamma(D))$. If $\eta(\Gamma(D))=1$, then it follows directly from Lemma \ref{labeling-a-cycle}. We now assume that $\eta(\Gamma(D))=r \geq 2$ and that the lemma holds whenever $\Gamma(D) < r$. By Lemma \ref{inductive construction net}, the graph $\Gamma(D)$ can be constructed iteratively in $r$ steps by taking a simple cycle having two non $\mathcal{A}$-points in the first step and then gluing the endpoints of a segment to two different points at each step. Let $D_1: A \to R \cup \{\epsilon\}$ and $D_2: A \to R \cup \{\epsilon\}$ be two functions such that $\Gamma(D_2)$ is the subgraph of $\Gamma(D)$, obtained after $r-1$ steps of such an iterative construction of $\Gamma(D)$, and $\Gamma(D_1):u_1 \to u_2 \to \cdots \to u_m$ is the segment glued at the last step to the net $\Gamma(D_2)$, to finally obtain $\Gamma(D)$.

We choose an anchor $\mathcal{A}'$ of $\Gamma(D_2)$ such that $\Gamma(D_2) \cap\mathcal A \subseteq \mathcal{A}'$. This choice can be made because $\Gamma(D_2) \cap \mathcal A$ is an admissible set of vertices in $\Gamma(D_2)$. Since $\eta(\Gamma(D_2)) = r-1$, the lemma holds for $\Gamma(D_2)$ by induction hypothesis. Let $v_{t_1}$ and $v_{t_2}$ be the vertices of $\Gamma(D_2)$ that are glued, respectively, to the end points $u_1$ and $u_m$ of $\Gamma(D_1)$ at the last step of iterative construction to obtain $\Gamma(D)$. By Lemma \ref{inductive construction net}, it can be assumed that $\Gamma(D_1)$ has either one or two non $\mathcal{A}$-points.
    
We first assume that $v_s$ is a vertex of $\Gamma(D_2)$. Further, assume that there is a unique non $\mathcal{A}$-point $u_i$ of $\Gamma(D)$ that is contained in $\Gamma(D_1)$. By Lemma \ref{net parity}, the vertices $v_{t_1}$ and $v_{t_2}$ are of opposite $\sigma$-parity in $\Gamma(D)$, the vertices $v_{t_1}, u_2, u_3, \dots, u_{i-1}$ have the same $\sigma$-parity in $\Gamma(D)$, and the vertices $u_{i+1}, u_{i+2}, \dots, u_{m-1}, v_{t_2}$ have the same $\sigma$-parity in $\Gamma(D)$.
This is because the vertices $u_1$ and $u_m$ are glued to $v_{t_1}$ and $v_{t_2}$, respectively, and $u_i$ is the unique non $\mathcal{A}$-point in $\Gamma(D_1)$. By induction, $\Gamma(D_2)$ admits a consistent labeling $v_k \mapsto (\alpha_k, \beta_k)$ such that $x_s = a$, $y_s =b$, and the conditions $(i)$, $(ii)$, $(iv)$ and $(v)$ in the lemma are satisfied for all $\mathcal{A}'$-points. Without loss of generality, we assume that $v_{t_1}, u_2, u_3, \dots, u_{i-1}$ have positive $\sigma$-parity, and the vertices $u_{i+1}, u_{i+2}, \dots, u_{m-1}, v_{t_2}$ have negative $\sigma$-parity in $\Gamma(D)$. Thus, the consistent labeling of $\Gamma(D_2)$ asserts that $[\alpha_{t_1}] \neq [0]$ and $[\beta_{t_2}] \neq [0]$.

Now, we proceed with extending this labeling to $\Gamma(D)$. We apply an idea
similar to Step $1$ of Lemma \ref{labeling-a-cycle}, for labeling $S_1 = \{v_{t_1}, u_2, u_3, \dots, u_{i-1}\}$. For labeling $S_2 = \{v_{t_2}, u_{i+1}, u_{i+2}, \dots, u_{m-1}\}$, we define an iterative sequence similar to the labeling of the vertices $v_{r+1}, v_{r+2}, \cdots, v_{n-1}$ in Step $2$ of Lemma \ref{labeling-a-cycle}.
Finally, using the idea as in Step $4$ of \ref{labeling-a-cycle}, we label the vertex $u_{i}$.

We now assume that $u_i, u_j, i<j$ are two non $\mathcal{A}$-points of $\Gamma(D)$ that are contained in $\Gamma(D_1)$. By Lemma \ref{net parity}, the vertices $v_{t_1}$ and $v_{t_2}$ are of the same $\sigma$-parity in $\Gamma(D)$. Without loss of generality, assume that they are of positive $\sigma$-parity. In this case, we again apply induction on $\Gamma(D_2)$, and obtain a consistent labeling $v_k \mapsto (\alpha_k, \beta_k)$ on $\Gamma(D_2)$ such that $x_s = a$ and $y_s =b$ and the conditions $(i)$, $(ii)$, $(iv)$ and $(v)$ in the lemma are satisfied for all $\mathcal{A}_1$-points. Thus, $[\alpha_{t_1}] \neq [0]$, $[\alpha_{t_2}] \neq [0]$. We proceed similar to Lemma \ref{labeling-a-cycle}, by taking $S_1 = \{v_{t_1}, u_2, u_3, \dots, u_{i-1}\}$, $S_2 = \{ u_i, u_{i+1}, \dots, u_{j-1} \}$, $S_3 = \{u_{j+1},  u_{j+2}, \dots, u_{m-1}, v_{t_2}\}$ and $S_4 = \{u_j\}$.
This addresses the case when $v_s$ is a vertex of $\Gamma(D_2)$.

Now, assume that $v_s$ is a vertex of $\Gamma(D_1)$. We first consider the case when $\Gamma(D_1)$ has a unique non $\mathcal{A}$-point $u_i$. By Lemma \ref{net parity}, the vertices $v_{t_1}$ and $v_{t_2}$ are of opposite $\sigma$-parity in $\Gamma(D)$. Without loss of generality, we assume that $v_s = u_\ell$, for some $\ell>i$. Thus, the vertices $v_{t_1}, u_2, u_3, \cdots, u_{i-1}$ have negative $\sigma$-parity and the vertices $u_\ell, u_{\ell+1}, \cdots, u_{m-1}, v_{t_2}$ have positive $\sigma$-parity. We first label the set $S_1 = \{ u_\ell, u_{\ell+1}, \cdots, u_{m-1}, v_{t_2}\}$ along the procedure in Step $1$ of Lemma \ref{labeling-a-cycle}. Consequently, let $(\alpha_{t_2}, \beta_{t_2})$ be the label thus assigned to $v_{t_2}$. By fixing the solution pair $(\alpha_{t_2}, \beta_{t_2})$ for $v_{t_2}$, we apply induction on $\Gamma(D_2)$ and obtain a consistent labeling for it that is in accordance with the conditions $(i)$, $(ii)$, $(iv)$ and $(v)$ of the lemma. Let this labeling assigns $(\alpha_{t_1}, \beta_{t_1})$ with $[\beta_{t_1}] \ne [0]$ as a solution pair for $v_{t_1}$. We then fix it and proceed by taking $S_2 = \{v_{t_1}, u_{2}, \cdots, u_{i-1}\}$ similar to the labeling of vertices $v_{r+1}, v_{r+2}, \cdots, v_{n-1}$ in Step $2$ of Lemma \ref{labeling-a-cycle}. Finally, we proceed similar to Step $3$ and Step $4$ of Lemma \ref{labeling-a-cycle} by taking $S_3 = \{u_{i+1}, u_{i+2}, \cdots, u_{\ell-1} \}$ and $S_4 = \{u_i\}$ to conclude the lemma in this case.

Now, we assume that $u_i, u_j, i<j$ are two non $\mathcal{A}$-points of $\Gamma(D)$ that are contained in $\Gamma(D_1)$. By Lemma \ref{net parity}, the vertices $v_{t_1}$ and $v_{t_2}$ are of the same $\sigma$-parity in $\Gamma(D)$. Let $v_s = u_\ell$ for some $i < \ell < j$. Then $v_{t_1}$ and $v_{t_2}$ are of negative $\sigma$-parity. Using the same idea as in Step $1$ and Step $2$ of Lemma \ref{labeling-a-cycle}, by taking $S_1 = \{u_\ell, u_{\ell + 1}, \cdots, u_{j-1}\}$ and $S_2 = \{u_j, u_{j + 1}, \cdots, u_{m-1}, v_{t_2}\}$, we label vertices in $S_1 \cup S_2$. Through this labeling, let $(\alpha_{t_2}, \beta_{t_2})$ be the solution pair for $v_{t_2}$. By fixing this solution pair for $v_{t_2}$ and applying induction on $\Gamma(D_2)$ for the restriction of the sign function $\sigma^{-}$, we obtain a consistent labeling for $\Gamma(D_2)$ that is in accordance with the conditions $(i)$, $(ii)$, $(iv)$ and $(v)$ of the lemma. Let this labeling assigns $(\alpha_{t_1}, \beta_{t_1})$ as a solution pair for $v_{t_1}$. Using the idea of Step $1$, Step $3$ and Step $4$ of Lemma \ref{labeling-a-cycle}, by taking $S_1 = \{v_{t_1}, u_{2}, \cdots, u_{i-1}\}$, $S_3 = \{u_{i+1}, u_{i+2}, \cdots, u_{\ell-1} \}$ and $S_4 = \{u_i\}$.

Finally, let $v_s = u_\ell$ for some $\ell < i < j$ or $ i < j < \ell$. Without loss of generality, let $\ell < i < j$. Then $v_{t_1}$ is of positive $\sigma$-parity and $v_{t_2}$ is of negative $\sigma$-parity. We proceed as in Step $1$, Step $2$ and Step $3$ of Lemma \ref{labeling-a-cycle} by taking $S_1 = \{u_\ell, u_{\ell + 1}, \cdots, u_{i-1}\}$, $S_2 = \{u_i, u_{i + 1}, \cdots, u_{j-1}\}$ and $S_3 = \{u_\ell-1, u_{\ell -2}, \cdots, u_{2}, v_{t_1}\}$ to label vertices in $S_1 \cup S_2 \cup S_3$. Let $(\alpha_{t_1}, \beta_{t_1})$ be the solution pair thus obtained for $v_{t_1}$. We fix this solution pair for $v_{t_1}$ and apply induction on $\Gamma(D_2)$ to obtain a consistent labeling for $\Gamma(D_2)$ that is in accordance with the conditions $(i)$, $(ii)$, $(iv)$ and $(v)$ of the lemma. Let this labeling assigns $(\alpha_{t_2}, \beta_{t_2})$ as a solution pair for $v_{t_2}$. Finally, we proceed similar to Step $3$ and Step $4$ of Lemma \ref{labeling-a-cycle} by taking $S_3 = \{u_{j+1}, u_{j+2}, \cdots, u_{m -1}, v_{t_2}\}$ and $S_4 = \{u_j\}$ to conclude the lemma in this case as well.
\end{proof}

The following theorem culminates previous lemmas and provides a consistent labeling on any graph that is not a triangle and does not contain bad cycles.

\begin{theorem}\label{without bad cycle}
Let $R$ be a local ring with at least three residue classes modulo its maximal ideal $\mathfrak m$ and $D: A \to R \cup \{\epsilon\}$ be a function such that the graph $\Gamma(D)$ is not a triangle and does not contain any bad cycle. Then $\Gamma(D)$ admits a consistent labeling.
\end{theorem}

\begin{proof}
 Let $\mathcal{A}$ be an anchor for the graph $\Gamma(D)$. Let $v_s \in \mathcal{A}$ be a fixed vertex. We show that for any $(\gamma, \delta) \in R$ with $([\gamma], [\delta]) \ne ([0],[0])$, there exists a consistent labeling $v_k \mapsto (\alpha_k, \beta_k)$ such that $(\alpha_s, \beta_s) = (\gamma, \delta)$, and for each $v_i \in \mathcal{A}$, we have $([\alpha_i], [\beta_i]) \ne ([0],[0])$.
 
 If $\Gamma(D)$ is a borderless graph, then the theorem follows by Lemma \ref{borderless}. If $\Gamma(D)$ is a net, then the theorem follows by Lemma \ref{net}. Thus, we can assume that $\Gamma(D)$ is a connected graph which properly contains a net $\Gamma'$ with $\eta(\Gamma') > 1$. We now proceed by induction on the number of edges in $\Gamma(D)$. Let us denote it by $t$. 
Observe that all connected graphs with $t \leq 7$ are either borderless graphs or nets, or have a bad cycle. 
 
We thus assume that $t \geq 8$ and the theorem holds for all the functions $D': A' \to R \cup \{\epsilon\}$, for which $\Gamma(D')$ does not contain a bad cycle and the number of edges in $\Gamma(D')$ is less than $t$. Let $\Gamma_1$ be a net contained in $\Gamma(D)$ such that $\Gamma_1$ is not properly contained in any other net of $\Gamma(D)$. Since $\Gamma_1 \subsetneq \Gamma(D)$, there exists a vertex $v_i$ of $\Gamma_1$ which is adjacent to some vertex that is not contained in $\Gamma_1$. We claim that if $v_j$ is a vertex of $\Gamma(D)$ that is connected to $v_i$ through a simple path lying completely outside the net $\Gamma_1$, then $v_j$ cannot be connected to any other vertex of the net $\Gamma_1$ through a simple path lying completely outside the net $\Gamma_1$. Let $v_i$ and $v_j$ be connected through a simple path $v_i \to u_1 \to u_2 \to \dots  \to u_{r_1} \to v_j$ lying outside the net $\Gamma_1$ and $v_\ell$ be some other vertex of $\Gamma_1$, that is connected to $v_j$ through a simple path $v_j \to u'_1 \to u'_2 \to \dots \to u'_{r_2} \to v_\ell$ lying completely outside $\Gamma_1$. Thus, $v_i \to u_1 \to \dots \to u_{r_1} \to v_j \to u'_1 \to \dots \to u'_{r_2} \to v_\ell$ is a segment whose edges lie outside the net $\Gamma_1$, and it is connecting two vertices $v_i$ and $v_\ell$ of $\Gamma_1$. This contradicts our assumption that $\Gamma_1$ is not properly contained in any other net of $\Gamma(D)$ and hence establishes the claim.

Thus, there exist two subgraphs $\Gamma_2$ and $\Gamma_3$ of $\Gamma(D)$, that have a unique vertex, say $v_i$, in common and $\Gamma_1 \subseteq \Gamma_2$. We first assume that $v_s$ is a vertex of $\Gamma_2$. By using induction, we label all vertices of $\Gamma_2$ by assigning $x_s = \gamma$ and $y_s = \delta$ in such a way that either $[\gamma] \ne [0]$ or $[\delta] \ne [0]$. In this process, let the solution pair assigned to $v_i$ be $(\alpha_i, \beta_i)$. We fix this solution pair for $v_i$ and apply induction on $\Gamma_3$ to label all its vertices. This proves the theorem in this case. The case when $v_s$ is a vertex of $\Gamma_3$ follows similarly.  
\end{proof}

The following theorem does not require the assumption $|R/\mathfrak m| > 2$. However, it puts more restrictions on $\Gamma(D)$ in order to provide a consistent labeling on it.

\begin{theorem}\label{Arbitrary local rings}
Let $R$ be a local ring and $D: A \to R \cup \{\epsilon\}$ be a function such that the graph $\Gamma(D)$ is not a triangle and does not contain a bad cycle. Let $\mathcal{A}$ be an anchor for $\Gamma(D)$. Then $\Gamma(D)$ admits a consistent labeling, provided that the following conditions are held simultaneously.
\begin{enumerate}
\item[(i)] All the cycles in $\Gamma(D)$ share a common vertex $v_i$.
\item[(ii)] Any net in $\Gamma(D)$ can be constructed iteratively by gluing a segment containing a non $\mathcal{A}$-point of $\Gamma(D)$ adjacent to $v_i$.
\end{enumerate}
\end{theorem}

\begin{proof}
If $\Gamma(D)$ contains no cycle, then $\Gamma(D)$ is a tree and hence from Lemma \ref{labeling-a-tree}, $\Gamma(D)$ admits a consistent labeling. Thus, we assume that $\Gamma(D)$ contains cycles. Let $v_i$ be the common vertex of these cycles, as in the condition $(i)$. We label this vertex $v_i$ by $(\alpha_i, 0)$, where $\alpha_i$ is an arbitrary invertible element of $R$. Now, let $\mathfrak{C}$ be the collection of all cycles in $\Gamma(D)$ that do not share a vertex other than $v_i$ with any other cycle. We label all vertices of such cycles, following the labeling scheme of Corollary \ref{labeling-a-cycle-p=2}. It is evident from the definition of $\mathfrak{C}$ that if a cycle $C$ is in $\mathfrak{C}$, then any vertex $v_j \ne v_i$ of $C$ has either degree equal to $2$, or a tree $T_j$ is glued to $v_j$. The tree $T_j$, if it exists, consists of all vertices adjacent to $v_j$ in $\Gamma(D)$, except for the two adjacent vertices of $v_j$ in $C$. The vertices of all such trees can be labeled using the labeling scheme of Lemma \ref{labeling-a-tree}.

Now, let $\mathfrak{C'}$ be the collection of all cycles in $\Gamma(D)$ that share a segment with some other cycle in $\Gamma(D)$. Let $\Gamma_1$ be a net that can be obtained as a union of cycles from $\mathfrak{C'}$ and is not contained in any other net of $\Gamma(D)$. We wish to label the vertices of $\Gamma_1$ following the labeling scheme of Lemma \ref{net}. By using our hypothesis, we note that at each step of the iterative construction of the net $\Gamma_1$, the segment to be glued, say $\Gamma_r$, has a non $\mathcal{A}$-point, say $v_r$, adjacent to $v_i$. Hence, at each step, all the vertices of $\Gamma_r$ can be labeled as per the labeling scheme of Lemma \ref{net} by incorporating the labeling scheme of Corollary \ref{labeling-a-cycle-p=2} to label $v_r$. This enables us to consistently label all such nets that contain the vertex $v_i$ and as a consequence, we label all cycles in the collection $\mathfrak{C'}$.

Finally, all trees glued to $v_i$ or to any vertex of a net in $\Gamma(D)$, can be labeled according to the labeling scheme of Lemma \ref{labeling-a-tree}. This exhausts all vertices of $\Gamma(D)$ and thus we obtain a consistent labeling on $\Gamma(D)$.
\end{proof}

\section{Commutators and Commutator Subgroup of Nilpotent Groups of Class 2}\label{section for results on p-groups}
 
Let $G$ be a finite nilpotent $p$-group of class $2$. Let $Z(G)$ be its center and $G'$ be its derived subgroup. 
Let $g_1, g_2, \dots, g_m \in G$ be such that $B_G := \{g_i Z(G) : 1 \leq i \leq m\}$ is a generating set of the factor group $G/Z(G)$. Let $g, h \in G$ be such that $g = \prod_{i=1}^{m} g_i^{\alpha_i}z_1$ and $h = \prod_{i=1}^{m} g_i^{\beta_i}z_2$, where $z_1, z_2 \in Z(G)$ and $\alpha_i, \beta_i \in \mathbb{Z}$ for $1 \leq i \leq m$. Then

$$[g, h] = \left[\prod_{i=1}^{m} g_i^{\alpha_i}z_1, \prod_{i=1}^{m} g_i^{\beta_i}z_2\right] = \left[\prod_{i=1}^{m} g_i^{\alpha_i}, \prod_{i=1}^{m} g_i^{\beta_i}\right]= \prod_{1 \leq i<j \leq m} \left[g_i, g_j\right]^{\alpha_i \beta_j - \alpha_j \beta_i}.$$

Hence $G'$ is generated by $\{\left[g_i,g_j\right] : 1 \leq i < j \leq m\}$. Let $g \in G'$. Then there exist $d_{i,j} \in \mathbb{Z}$ such that
$$\prod_{1 \leq i<j \leq m} [g_i, g_j]^{d_{i,j}}.$$ Note that the choice of integers $d_{i,j}$'s is not unique. Let $$\mathcal{I} := \{i : d_{i,j}\ne 0 \text{ for some $j$}\} \cup \{j : d_{i,j}\ne 0 \text{ for some $i$}\}.$$ Let $|\mathcal{I}| = n$. We permute the indices $\{1,2, \cdots, m\}$ so that $\mathcal{I} = \{1, 2, \cdots , n\}$. Recall the notation: 
$$A(n) := \{(i,j) : 1 \leq i < j \leq n\},$$ and define the function $D : A(n) \to \mathbb{Z}_p \cup \{\varepsilon\}$ as follows: 

$$D(i,j) = \begin{cases}
  \varepsilon, ~~~~\quad \text{ if } [g_i, g_j] = 1,\\
  d_{i,j}, \quad  \text{if } [g_i, g_j] \ne 1.
\end{cases}$$

The codomain of this function is $\mathbb Z_p \cup \{\varepsilon\}$, where $\mathbb Z_p$ is the ring of $p$-adic integers. Here, we are regarding integers $d_{i,j}$ as $p$-adic integers. We wish to solve balance equations corresponding to $D$ over the local ring $\mathbb Z_p$. The solutions \emph{modulo} $\exp(G)$ of these equations can be used to write $g$ as a commutator. Note that the function $D$ depends on the generating set $B_G$, the element $g$, the choice of $d_{i,j}$, and the permutation that sorts out elements of $\mathcal I$ as first $n$ indices. We call such a function a \emph{presentation} of $g$. 

For a presentation $D$ of $g$, a weighted graph $\Gamma(D)$ and a system of balance equations can be attributed as per the discussion in $\S$ \ref{Section on balance equations}. Further, it is evident that $g$ is a commutator in $G$ if and only if there exists a presentation $D$ for which the graph $\Gamma(D)$ has a consistent labeling. Thus, the following two theorems follow directly from Lemma \ref{consistent labeling for a triangle}, Theorem \ref{without bad cycle}, and Theorem \ref{consistent labeling doesn't exist}.

\begin{theorem}\label{without bad cycle for p-groups elementwise}
Let $p \ne 2$ and $G$ be a $p$-group of nilpotency class $2$. Let $g \in G'$ be such that the graph $\Gamma(D)$ does not contain bad cycles for some presentation $D$ of $g$. Then $g$ is a commutator in $G$.
\end{theorem}

\begin{theorem}\label{Not in the image of commutator map on p-groups}
Let $G$ be a $p$-group of nilpotency class $2$. Let $g \in G'$ be such that the graph $\Gamma(D)$ contains a bad cycle with unfavorable proximity for each choice of $D$. Then $g$ is not a commutator in $G$.
\end{theorem}

The following theorem covers a smaller family of groups when compared with Theorem \ref{without bad cycle for p-groups elementwise} but also holds when $p = 2$.

\begin{theorem}\label{Arbitrary p-groups elementwise}
Let $G$ be a $p$-group of nilpotency class $2$. Let $g \in G'$ and $D$ be a presentation of $G$ such that the graph $\Gamma(D)$ does not contain bad cycles. Let $\mathcal{A}$ be an anchor for $\Gamma(D)$. Then $\Gamma(D)$ admits a consistent labeling, provided the following conditions hold simultaneously.
\begin{enumerate}
\item[(i)] All cycles in $\Gamma(D)$ share a common vertex $v_i$.
\item[(ii)] Any net in $\Gamma(D)$ can be constructed iteratively by gluing a segment containing a non $\mathcal{A}$-point of $\Gamma(D)$ adjacent to $v_i$. 
\end{enumerate}

\begin{proof}
 It follows directly from Lemma \ref{consistent labeling for a triangle} and Theorem \ref{Arbitrary local rings}.   
\end{proof}
\end{theorem}
We remark that Theorem \ref{without bad cycle for p-groups elementwise}, Theorem \ref{Not in the image of commutator map on p-groups} and Theorem \ref{Arbitrary p-groups elementwise} hold for infinite $p$-groups of nilpotency class $2$ as well. However, for the rest of this section, we need to assume that $G$ is a $p$-group for which $B_G$ is finite. We can also assume that $B_G$ is a minimal generating set of $G/Z(G)$. We construct a graph $\Gamma(B_G)$ as follows.

\begin{enumerate}
\item The vertices of $\Gamma(B_G)$ are enumerated by the elements of $B_G$ whose representatives in $G$ do not commute with a representative of some other element of $B_G$. The vertex corresponding to $\overline{g_i}$ is denoted by $v_i$. Thus,
$$
V(B_G) := \{v_i \in B_G : [g_i, g_j] \neq 1 \text{ for some } \overline{g_j} \in B_G\}
$$
is the vertex set of $\Gamma(B_G)$.

\item For $i < j$, the vertices $v_i$ and $v_j$ are connected through an edge $e_{i,j}$ in $\Gamma(B_G)$ if $[g_i,g_j] \neq 1$. The edge set of $\Gamma(B_G)$ is denoted by $E(B_G)$.
\end{enumerate}

The following is a direct consequence of Theorem \ref{without bad cycle for p-groups elementwise}.

\begin{corollary}\label{Without bad cycle for whole group}
Let $p \ne 2$. Let $G$ be a $p$-group of nilpotency class $2$ and $B_G$ be a generating set of $G/Z(G)$. If $\Gamma(B_G)$ does not contain bad cycles, then $K(G) = G'$.  
\end{corollary}

\begin{proof}
Since the generating set $B_G$ is fixed, the graph $\Gamma(D)$ is a subgraph of $\Gamma(B_G)$ for any $g \in G'$ and its presentation $D$, thus the Corollary follows.
\end{proof}

Similarly, we have the following as a direct consequence of Theorem \ref{Arbitrary p-groups elementwise}.

\begin{corollary}\label{Arbitrary $p$-group}
Let $G$ be a $2$-group of nilpotency class $2$, $B_G$ be a generating set of $G/Z(G)$ such that $\Gamma(B_G)$ does not contain bad cycles. Let $\mathcal{A}$ be an anchor for $\Gamma(B_G)$. Then $K(G) = G'$, provided the following conditions hold simultaneously.
\begin{enumerate}
\item[(i).] All cycles in $\Gamma(B_G)$ share a common vertex $v_i$.
\item[(ii).] Any net in $\Gamma(B_G)$ can be constructed iteratively by gluing a segment containing a non $\mathcal{A}$-point of $\Gamma(D)$ adjacent to $v_i$.
\end{enumerate} 
\end{corollary}

\begin{remark}\label{Remark showing p=2 case is same for small groups}
There are graphs that do not contain a bad cycle but fail to satisfy the condition $(i)$ of Corollary \ref{Arbitrary $p$-group}. Note that there is no such connected graph with the number of vertices to be $7$ or less. The graph illustrated in Figure \ref{p=2 case condition (i)} is one such with $8$ vertices. Thus, to check whether $K(G) = G'$ for $2$-groups whose graphs $\Gamma(B_G)$ do not contain a bad cycle, we are required to investigate conditions $(i)$ and $(ii)$ of Corollary \ref{Arbitrary $p$-group} only when $|G/Z(G)| \geq 2^8$. 

Similarly, some graphs do not contain a bad cycle, satisfy condition $(i)$ of \ref{Arbitrary $p$-group}, but fail to satisfy condition $(ii)$ of Corollary \ref{Arbitrary $p$-group}. Note that there is no such connected graph with the number of vertices to be $16$ or less. The graph illustrated in Figure \ref{p=2 case condition (ii)} is one such with $17$ vertices. Thus, to check whether $K(G) = G'$ for $2$-groups whose graph $\Gamma(B_G)$ does not contain a bad cycle and satisfies condition (i), we are required to investigate condition $(ii)$ only when $G/Z(G) \geq 2^{17}$. 

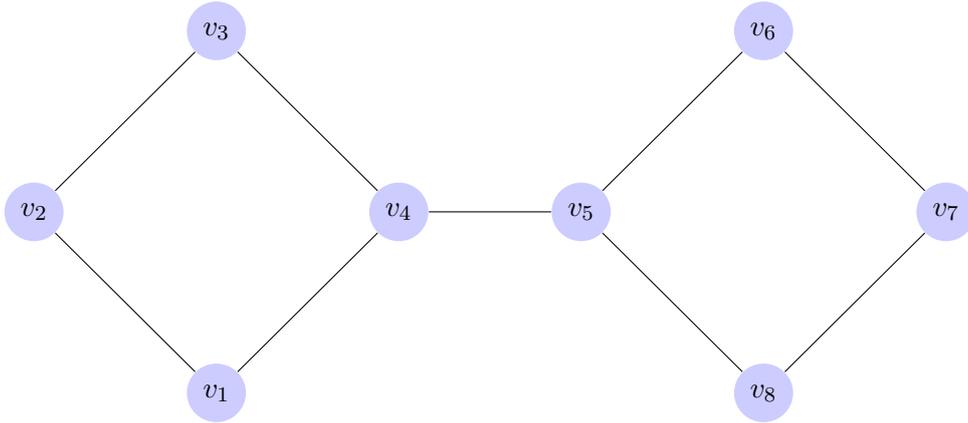
\begin{figure}
\begin{tikzpicture}
  [scale=.6,auto=left,every node/.style={circle,fill=blue!20}]
  \node (n1) at (5,4)   {$v_1$};
  \node (n2) at (1,8)   {$v_2$};
  \node (n3) at (5,12)  {$v_3$};
  \node (n4) at (9,8)   {$v_4$};
  \node (n5) at (13,8)  {$v_5$};
  \node (n6) at (17,12) {$v_6$};
  \node (n7) at (21,8)  {$v_7$};
  \node (n8) at (17,4)  {$v_8$};
  
  \foreach \from/\to in {n1/n2,n2/n3,n3/n4,n4/n1,n4/n5,n5/n6,n6/n7,n7/n8,n8/n5}
      \draw (\from) -- (\to);
\end{tikzpicture}
\caption{A graph that contains no bad cycles but fails to satisfy the condition $(i)$ of Corollary \ref{Arbitrary $p$-group}.
\label{p=2 case condition (i)}
}
\end{figure}

\begin{center}
\begin{figure}[ht]
\begin{tikzpicture}
  [scale=1.0,auto=left,every node/.style={circle,fill=blue!20}]
  \node (n1) at (4,10)      {$v_1$};
  \node (n2) at (7,10)      {$v_2$};
  \node (n3) at (10,10)     {$v_3$};
  \node (n4) at (12,7)      {$v_4$}; 
  \node (n5) at (10,4)      {$v_5$};
  \node (n6) at (7,4)       {$v_6$};
  \node (n7) at (4,4)       {$v_7$};
  \node (n8) at (1,7)       {$v_8$};
  \node (n9) at (1,13)      {$v_9$};
  \node (n10) at (13,13)    {$v_{10}$};
  \node (n11) at (13,1)     {$v_{11}$};
  \node (n12) at (1,1)      {$v_{12}$}; 
  \node (n13) at (8.5,8.5)  {$v_{13}$};
  \node (n14) at (10,7)     {$v_{14}$};
  \node (n15) at (7,7)      {$v_{15}$};
  \node (n16) at (5.5,5.5)  {$v_{16}$};
  \node (n17) at (4,7)      {$v_{17}$};
  
  \foreach \from/\to in {n1/n2,n2/n3,n3/n4,n4/n5,n5/n6,n6/n7,n7/n8,n8/n1,n1/n9,n3/n10,n5/n11,n7/n12,n2/n13,n13/n14,n13/n15,n15/n16,n16/n6,n16/n17}
      \draw (\from) -- (\to);
\end{tikzpicture}
\caption{A graph that contains no bad cycles, satisfies the condition $(i)$ of Corollary \ref{Arbitrary $p$-group} but fails to satisfy the condition $(ii)$ of Corollary \ref{Arbitrary $p$-group}.}\label{p=2 case condition (ii)}
\end{figure}
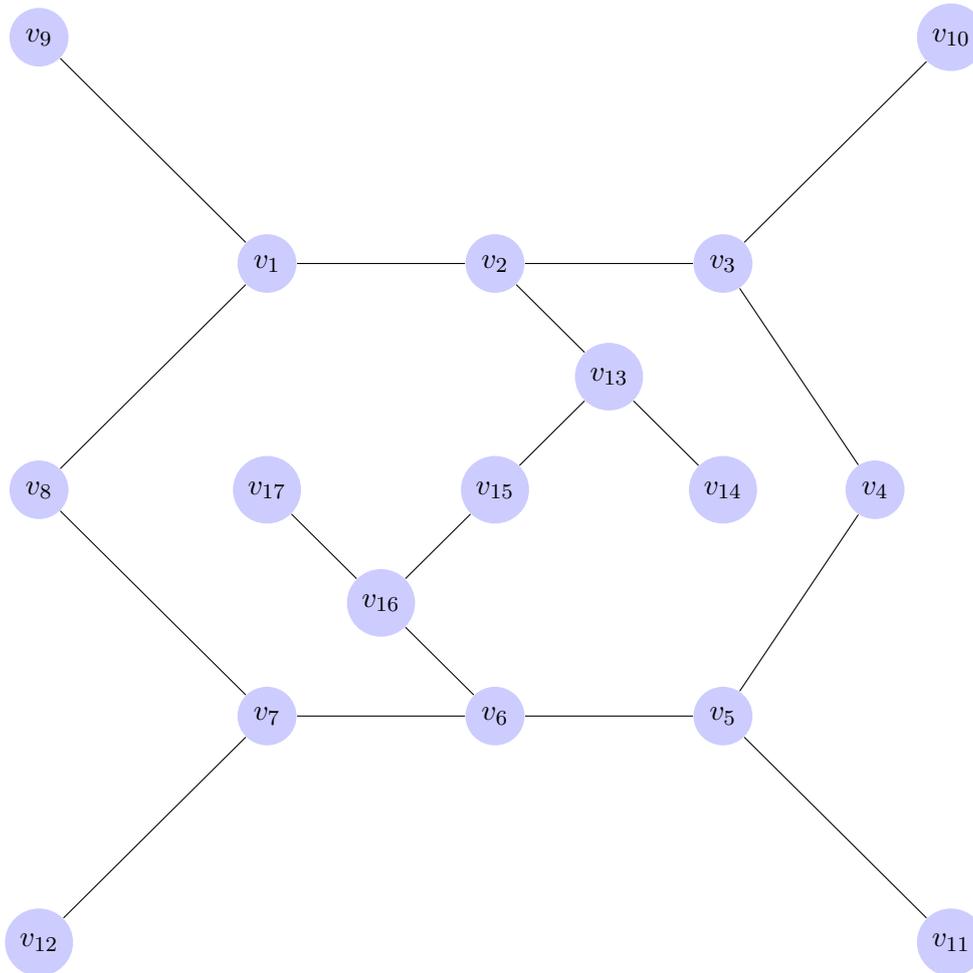
\end{center}
\end{remark}

The following theorem gives a necessary and sufficient condition for $K(G) = G'$ in a specific case.
\begin{theorem}\label{iff condition for p-groups}
Let $p \ne 2$ and $G$ be a $p$-group of nilpotency class $2$. Let $B_G:= \{g_iZ(G): 1\leq i \leq n\}$ be a generating set of $G/Z(G)$ and $\Gamma(B_G)$ be the associated graph. Suppose, the set $E_G:=\{[g_i,g_j]: \text{ $e_{i,j}$ is an edge in $\Gamma(B_G)$} \}$ forms a minimal generating set of $G'$. Then $K(G) = G'$ if and only if $\Gamma(B_G)$ does not contain bad cycles.
\end{theorem}

\begin{proof}
If $\Gamma(B_G)$ does not contain a bad cycle, then the result follows by Corollary \ref{Without bad cycle for whole group}. For the converse, let $C_r$ be a bad cycle in $\Gamma(B_G)$, involving $r$ vertices. 
Let $\mathcal{P} \subseteq \Gamma(B_G)$ be a proximity for the cycle $C_r$. We enumerate the vertices of $\mathcal{P}$ by $v_1, v_2, \dots, v_{2r-2}$ such that the vertices $v_1, v_2, \dots, v_{r}$ form the cycle $C_r$ and the vertex $v_{r+\ell}$ is adjacent to the vertex $v_\ell$ in $\mathcal{P}$, where $\ell \in \{1,2,\dots, r-2\}$. For any $1 \leq i <j \leq n$, let $c_{i,j}$ denote the commutator $[g_i, g_j]$ in $G$.
Denote $S := \{(i,r+i) : 1\leq i \leq r-2\} \cup \{(r-1,r)\}$,
and $$g :=\prod_{(i,j) \in S} c_{i,j}.$$
It is evident that $g \in G'$. We show that $g \notin K(G)$. For the above expression of $g$, the presentation function $D:A(n) \to \mathbb Z_p \cup \{\varepsilon\}$ is given by
$$D(i,j) = \begin{cases}
  1, ~~~\quad \text{ if } (i,j) \in S,\\
  \varepsilon, ~\quad  \text{if } [g_i, g_j] =1,\\
  0, \quad  \text{ otherwise }.\\
\end{cases}$$

By definition, the graph $\Gamma(D)$ contains a bad cycle with unfavorable proximity. We claim that the graph of any presentation function of $g$ contains a bad cycle with unfavorable proximity.

Let $$g = \prod_{1 \leq i < j \leq n} c_{i,j}^{d_{i,j}},$$
where $d_{i,j} \in \mathbb Z$, be some other expression of $g$. If possible, let $(s,t) \in S$ be such that
$d_{s,t} \not\equiv 1 \bmod p$. Then
$$c_{s,t}^{1-d_{s,t}} = g^{-1}gc_{s,t}^{1-d_{s,t}} = \left(\prod_{(i,j) \in S} c_{i,j}\right)^{-1}\left(\prod_{1 \leq i < j \leq n} c_{i,j}^{d_{i,j}}\right) c_{s,t}^{1-d_{s,t}}.$$
The exponent of $c_{s,t}$ in the right-hand side of the above equation is $0$ but $1 - d_{s,t}$ is coprime to $p$. Thus, $c_{s,t}$ belongs to the subgroup generated by $E_G\backslash \{c_{s,t}\}$. This contradicts the minimality of the set $E_G$. Thus, $d_{s,t} \equiv 1 \bmod p$ for all $(s,t) \in S$.

Now, let $p \nmid d_{s,t}$ for some $(s,t) \in \{(i,i+1) : 1 \leq i \leq r-2 \} \cup \{(1,r)\} \subseteq E_G \backslash S$. Then $$c_{s,t}^{d_{s,t}} = g^{-1} g c_{s,t}^{d_{s,t}} = \left(\prod_{1 \leq i < j \leq n} c_{i,j}^{d_{i,j}}\right)^{-1} \left(\prod_{(i,j) \in S} c_{i,j}\right)c_{s,t}^{d_{s,t}}.$$ The exponent of $c_{s,t}$ in the right hand side of the above equation is $0$ but $d_{s,t}$ is coprime to $p$ and hence this again contradicts the minimality of the set $E_G$. Thus, $p \mid d_{s,t}$ for all $(s,t) \in \{(i,i+1) : 1 \leq i \leq r-2 \} \cup \{(1,r)\}$. This shows that the expression $$g = \prod_{1 \leq i < j \leq n} c_{i,j}^{d_{i,j}}$$ also forces $\mathcal{P}$ to be an unfavorable proximity. Thus, we conclude that the graph $\Gamma(D)$ will always have an unfavorable proximity $\mathcal{P}$ for any given expression of $g$ in $G'$. Thus, by Theorem \ref{Not in the image of commutator map on p-groups}, $g \in G' \backslash K(G)$.
\end{proof}

We carefully follow the proof of Theorem \ref{iff condition for p-groups} to note that if $p=2$ and $E_G$ forms a minimal generating set of $G'$, then $K(G) \ne G'$ provided that the graph $\Gamma(B_G)$ contains a bad cycle. We do not have a proof of the converse for the $p=2$ case.

\begin{remark}\label{Remark for minimality condition for groups is necessary}
If the set $E_G$ is not a minimal generating set of $G'$, then the group $G$ may still have $K(G) =G'$, even if the graph $\Gamma(B_G)$ contains a bad cycle. One such example is 
\begin{align*}
G : = \langle &g_1, g_2, g_3, g_4 : g_i^p =1, [g_1, g_2] = [g_1, g_3] = [g_1, g_4] = [g_2, g_3], \\
&[g_2, g_4] = [g_3, g_4] = 1, [[g_1, g_2], g_i] = 1, \text{ for each } i \in \{1,2,3,4\} \rangle.
\end{align*}

Here, $Z(G)= G'$ is of order $p$ and $G$ is of order $p^5$. For $B_G := \{\overline{g_1}, \overline{g_2}, \overline{g_3}, \overline{g_4}\}$, the graph $\Gamma(B_G)$ is illustrated in Figure \ref{figure for minimality condition for groups is necessary}. Clearly, it contains a bad cycle. But $G'$ is cyclic of order $p$ and it is generated by $[g_1, g_2]$. Thus, each element $g$ of $G$ has a presentation function $D$ such that the graph $\Gamma(D)$ is a single edge $e_{1,2}$. Hence, $K(G) = G'$ by Theorem \ref{without bad cycle for p-groups elementwise}.
\end{remark}

\begin{center}
\begin{figure}[ht]
\begin{tikzpicture}
  [scale=0.8,auto=left,every node/.style={circle,fill=blue!20}]
  \node (n1) at (9,1)      {$v_1$};
  \node (n2) at (5,5)      {$v_2$};
  \node (n3) at (1,1)     {$v_3$};
  \node (n4) at (9,5)      {$v_4$}; 
  
  \foreach \from/\to in {n1/n2,n2/n3,n1/n3,n1/n4}
      \draw (\from) -- (\to);
\end{tikzpicture}
\caption{Graph for the group in Remark \ref{Remark for minimality condition for groups is necessary}}\label{figure for minimality condition for groups is necessary}
\end{figure}
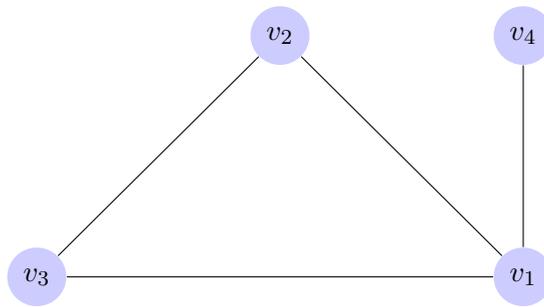
\end{center}

\subsection{Constructing a $p$-group from a graph}\label{Remark from graph to group}
Let $\Gamma:=(V, E)$ be a graph without isolated vertices. Let $V:= \{v_1, v_2, \dots , v_n\}$ and $F_n$ denote the free group generated by the free generators $x_1, x_2, \dots, x_n$.
For each $r \in \mathbb{N}$, we associate to $\Gamma$, the group $G_\Gamma:= F_n/R$, where $R$ is the normal subgroup generated by the following relations:

\begin{enumerate}
\item[$(i)$.] If $v_i$ and $v_j$ are not adjacent in $\Gamma$, then $[x_i, x_j] \in R$.
\item[$(ii)$.] If $v_i$ and $v_j$ are adjacent in $\Gamma$, then $[[x_i, x_j], x_k] \in R$ for every $k \in \{1,2,\cdots,n\}$.
\item[$(iii)$.] ${x_{1}}^{p^r}, {x_{2}}^{p^r} \in R$.
\item[$(iv)$.] ${x_i}^{p} \in R$, for each $i \in \{3, 4, \cdots, n\}$.
\end{enumerate}

We denote the image of $x_i$ in $G_{\Gamma}$ by $g_i$.
The group $G_\Gamma$ has the following properties:
\begin{enumerate}
\item[(a).] The set $B_{G_\Gamma} := \{g_i Z(G) : 1 \leq i \leq n\}$ is a generating set of the factor group $G_\Gamma/Z(G_\Gamma)$, and $\Gamma(B_{G_\Gamma}) = \Gamma$.
\item[(b).] The group $G_\Gamma$ is a nilpotent group of class $2$. This follows from conditions $(i)$ and $(ii)$. Therefore, $[g_{1}, g_{2}]^{p^r} = [g_{1}^{p^r}, g_{2}] =1$. Further, if $(i,j) \ne (1 , 2)$, then $[g_{i}, g_{j}]^{p} =1$.
\item[(c).] $Z(G_\Gamma) = G_\Gamma'$ and it is minimally generated by the set $\{[g_i, g_j] : e_{i,j} \in E\}$.
\item[(d).] $|G_\Gamma'| = |Z(G_\Gamma)| = p^{|E|+r-1}$, $|G_\Gamma/Z(G_\Gamma)| = p^{|V|+2r-2}$, $|G_\Gamma| = p^{|E|+|V|+3r-3}$, size of the minimal generating set of $G_\Gamma$ is $|V|$, size of the minimal generating set of $G_\Gamma'$ is $|E|$, and 
$$\exp(G_\Gamma) = \begin{cases}
  p^r, \quad\quad \text{ if } p \ne 2,\\
  2^{r+1}, ~~\quad  \text{if } p = 2.\\
\end{cases}$$
\end{enumerate}

The above construction of $G_{\Gamma}$ brings out the substance of Theorem \ref{iff condition for p-groups}. If $p \neq 2$, then by Theorem \ref{iff condition for p-groups}, determining whether or not $K(G_\Gamma)$ and $G_\Gamma'$ are equal, is a matter of locating a bad cycle in the graph $\Gamma$.

The following corollary guarantees the existence of a group $G$ with a prescribed order, an admissible exponent, and $K(G) \ne G'$. Since a connected graph $\Gamma = (V,E)$ containing a bad cycle satisfies 
$|V| \geq 4$ and $|E| \geq |V|$, we assume that for the groups $G$ with $K(G) \ne G'$, the associated graph $\Gamma(B_G) = (V,E)$ satisfies $|V| \geq 4$ and $|E| \geq |V|$. 

\begin{corollary}\label{Infinite examples for K(G) not equal to G'}
Let $r, s, t$ be positive integers such that $m := s-r+1$ and $n := t-s-2r+2$ satisfy $4 \leq n \leq m \leq \frac{1}{2}n (n -1)$. Then there exists a $p$-group $G$ of nilpotent class $2$ with the following property:
\begin{enumerate}
\item[(i).] $|G| = p^t$, $|Z(G)| = |G'| = p^s$.
\item[(ii).] ${\rm exp}(G) = \begin{cases}
  p^r, \quad\quad \text{ if } p \ne 2,\\
  2^{r+1}, ~~\quad  \text{if } p = 2.\\
\end{cases}$
\item[(iii).] $K(G) \neq G'$. 
\end{enumerate}
\end{corollary}
\begin{proof}
We first construct a connected simple graph $\Gamma = (V, E)$ with $n$ vertices and $m$ edges, that contains a bad cycle. We observe that the condition $n \leq m \leq \frac{1}{2}n (n -1)$ is necessary for this construction, otherwise, either $\Gamma$ won't be connected and simple, or it will be a tree. 

If $n=m$, then we construct $\Gamma$ as follows:
\begin{align*}
V &:= \{v_1, v_2, \cdots, v_n\},\\
E &:= \{e_{i,i+1}: 1\leq i \leq n-1\} \cup \{e_{1,3}\},    
\end{align*}
where $e_{i,j}$ is an edge between $v_i$ and $v_j$.
Since $n \geq 4$, this graph contains a bad cycle involving vertices $v_1, v_2, v_3$ and $v_4$. 
Let us denote this graph by $\Gamma_4$.

If $n < m$, then we add a sufficient number of edges to 
$\Gamma_4$, without adding a vertex, to obtain a graph $\Gamma$ that contains a bad cycle involving the vertices $v_1, v_2, v_3$ and $v_4$.

We then use $\S \ref{Remark from graph to group}$ to associate a group $G_\Gamma$ to the graph $\Gamma$. For $G = G_{\Gamma}$ we have:

\begin{align*}
|G| &= p^{|E|+|V|+3r-3} = p^{m+n+3r-3} = p^t, \\
|G'| &= |Z(G)| = p^{|E|+r-1} = p^{m+r-1} = p^s, \\
{\rm exp}(G) &= \begin{cases}
  p^r, \quad\quad \text{ if } p \ne 2,\\
  2^{r+1}, ~~\quad  \text{if } p = 2.\\
\end{cases}
\end{align*}
Since $\Gamma(B_{G}) = \Gamma$ contains a bad cycle, Theorem \ref{iff condition for p-groups} guarantees that 
$K(G) \ne G'$.
\end{proof}

\section{Surjectivity of bilinear maps and Lie bracket}\label{section on bilinear maps and Lie algebra}

Let $U$ and $W$ be two vector spaces over a field $F$, with countable Hamel bases. Let $B: U \times U \to W$ be an alternating bilinear map over $F$ and $B(U \times U)$ be the image of $B$. We assume that $W$ is equal to the subspace spanned by the set $B(U \times U)$. It is natural to ask when $B(U \times U)$ equals $W$. In this section, we address this question and then deal with the problem of determining the bracket width of Lie algebras. Our results on consistent labeling of graphs proved in \S\ref{Labeling of Graphs} will play a crucial role.

Let $U^\perp := \{v \in U : B(v, v') =0, \text{ for all } \ v' \in U\}$.  We fix a basis $\mathcal B_1:= \{v_1, v_2, \dots \}$ of $U^{\perp}$ and extend it to a basis $\mathcal{B}:= \mathcal B_1 \cup \mathcal B_2$ of $U$, where $\mathcal B_2 = \{u_1, u_2, \dots\}$. Let $v, v' \in U$. We write
$$v = \sum_{i=1}^{r} \alpha_i u_i + \sum_{i=1}^{s} \gamma_i v_i\quad \text{and} \quad v' = \sum_{i=1}^{r} \beta_i u_i + \sum_{i=1}^{s}\delta_i v_i,$$
where $r, s$ are suitable positive integers and $\alpha_i, \beta_i, \gamma_i, \delta_i \in F$ are suitable scalars.

Thus,
\begin{align*}
    B(v,v') &= B\left(\sum_{i=1}^{r} \alpha_i u_i + \sum_{i=1}^{s} \gamma_i v_i, \sum_{i=1}^{r} \beta_i u_i + \sum_{i=1}^{s}\delta_i v_i\right) \\
    &= B\left(\sum_{i=1}^{r} \alpha_i u_i, \sum_{i=1}^{r} \beta_i u_i \right) = \sum_{1 \leq i < j \leq r}(\alpha_i \beta_j - \alpha_j \beta_i) B(u_i, u_j).\\
\end{align*}  
Here, the second equality holds because $v_i \in U^{\perp}$, and the third equality holds because $B$ is an alternating bilinear map. Thus, $$W = {\rm span}(\{B(u_i, u_j) : u_i, u_j \in \mathcal{B}_2, i < j\}).$$ 
We write an arbitrary $w \in W$ as $$w = \sum_{1 \leq i < j \leq r} d_{i,j} B(u_i, u_j); \quad d_{i,j} \in F.$$ 
Note that in general, the set $\{B(u_i, u_j) : u_i, u_j \in \mathcal{B}_2\}$ need not be a basis of $W$. Thus, for any $w \in W$, multiple choices for the scalers $d_{i,j}$ may exist. Let $$\mathcal{I} := \{i : d_{i,j}\ne 0 \text{ for some $j$}\} \cup \{j : d_{i,j}\ne 0 \text{ for some $i$}\}.$$ Let $|\mathcal{I}| = n$. We permute the indices $\{1,2, \cdots, \}$ so that $\mathcal{I} = \{1, 2, \cdots , n\}$. Define the function $D : A(n) \to F \cup \{\varepsilon\}$ as follows: 

$$D(i,j) = \begin{cases}
  \varepsilon, ~~~~~~\quad \text{ if } B(u_i, u_j) = 0,\\
  d_{i,j}, \quad  \text{if } B(u_i, u_j) \neq 0.
\end{cases}$$

Note that the function $D$ depends on the set $\mathcal{B}_2$, the element $w$, the choice of $d_{i,j}$, and the permutation that sorts out elements of $\mathcal I$ as first $n$ indices. We call such a function a \emph{presentation} of $w$. The weighted graph $\Gamma(D)$ corresponds to a system of balance equations and it is clear that $w$ lies in the image of $B$ if and only if there exists a presentation $D$ of $w$ such that the corresponding graph $\Gamma(D)$ has a consistent labeling. The following two theorems follow directly from Lemma \ref{consistent labeling for a triangle}, Theorem \ref{without bad cycle}, and Theorem \ref{consistent labeling doesn't exist}.

\begin{theorem}\label{without bad cycle for bilinear maps}
Let $F \ne \mathbb{F}_2$ be a field and $U, W$ be two vector spaces over $F$ with countable Hamel bases. Let $B: U \times U \to W$ be an alternating bilinear map over $F$ such that ${\rm span}(B(U \times U)) = W$. Let $w \in W$ be such that for some presentation $D$ of $w$ the graph $\Gamma(D)$ does not contain bad cycles. Then $w \in B(U \times U)$.
\end{theorem}

\begin{theorem}\label{Not in the image of bilinear map}
Let $F$ be a field and $U, W$ be two vector spaces over $F$ with countable Hamel bases. Let $B: U \times U \to W$ be an alternating bilinear map over $F$ such that ${\rm span}(B(U \times U)) = W$. Let $w \in W$ be such that for each presentation $D$ of $w$, the graph $\Gamma(D)$ contains a bad cycle with unfavorable proximity. Then $w \notin B(U \times U)$.
\end{theorem}

Now, let us assume that $\mathcal B_2$ is a finite set and $|\mathcal B_2| = n$. Then we associate a graph $\Gamma(\mathcal B_U):= (V, E)$ to the vector space $U$ as follows:
\begin{align*}
&V := \{v_i : B(u_i, u_j) \neq 0 \text{ for some } u_j \in \mathcal{B}_2\},\\
&E:= \{e_{i,j} : B(u_i,u_j) \neq 0\},
\end{align*}
where $e_{i,j}$ is an edge between $v_i$ and $v_j$. The following theorem is a direct consequence of Theorem \ref{without bad cycle for bilinear maps}.

\begin{corollary}\label{Without bad cycle for whole vector space}
Let $F \ne \mathbb{F}_2$, $U$ be a vector space over $F$ and let $B: U \times U \to W$ be an alternating bilinear map over $F$ with finite-dimensional quotient space $U / U^{\perp}$. Let ${\rm span}(B(U \times U)) = W$ and $\mathcal{B}$ be a basis of $U / U^{\perp}$ such that the graph $\Gamma(\mathcal B_U)$ does not contain bad cycles, then $B(U\times U) =  W$.
\end{corollary}

The following theorem gives a necessary and sufficient condition for
$B(U\times U)) =  W$ in a specific case.

\begin{theorem}\label{iff condition for Bilinear maps}
Let $F \ne \mathbb{F}_2$, $U$ be a vector space over $F$ and let $B: U \times U \to W$ be an alternating bilinear map over $F$ with finite-dimensional quotient space $U / U^{\perp}$. Let ${\rm span}(B(U \times U)) = W$. Let $\mathcal{B}:= \{\overline{u_1}, \overline{u_2}, \cdots, \overline{u_n}\}$ be a basis of $U / U^{\perp}$ and let $\Gamma(\mathcal B_U) = (V,E)$ be the associated graph such that the set $\{B(u_i,u_j): e_{i,j} \in E\}$ forms a basis of $W$. Then $B(U \times U) = W$ if and only if $\Gamma(\mathcal B_U)$ does not contain bad cycles.
\end{theorem}

\begin{proof} The proof follows through the arguments similar to the proof of Theorem \ref{iff condition for p-groups}.
\end{proof}

A construction analogous to Remark \ref{Remark for minimality condition for groups is necessary} and $\S$\ref{Remark from graph to group} can be carried out for the case of bilinear maps as well. One may use it to construct infinitely many examples when $B(U \times U) \neq W$, a result that is analogous to Corollary \ref{Infinite examples for K(G) not equal to G'}.

The Lie bracket is an alternating bilinear map for a Lie algebra $L$ over a field $F$. Thus, to study images of Lie brackets using above approach, we take $U:= L$; $W:= L'$, the derived Lie subalgebra of $L$; $B = [~,~]$, the Lie bracket of $L$; $U^\perp = Z(L)$, the center of $L$; and $\mathcal{B}_L := \mathcal B_1 \cup \mathcal B_2$, the vector space basis of $L$; where $\mathcal B_1 $ is a basis of $Z(L)$ and $\mathcal B_2 = \{u_1, u_2 \cdots \}$. The following theorems are direct consequences of the above theorems on alternating bilinear maps.

\begin{theorem}\label{without bad cycle for Lie algebras}
Let $F \neq \mathbb{F}_2$ be a field and $L$ be a Lie algebra over $F$ having a countable Hamel basis. Let $x \in L'$ be such that for some presentation $D$ of $x$, the graph $\Gamma(D)$ does not contain bad cycles. Then $x \in [L, L]$.
\end{theorem}

\begin{theorem}\label{Not in the image of Lie bracket}
Let $F \neq \mathbb{F}_2$ be a field and $L$ be a Lie algebra over $F$ having a countable Hamel basis. Let $x \in L'$ be such that for each presentation $D$ of $x$, the graph $\Gamma(D)$ contains a bad cycle with unfavorable proximity. Then $x \notin [L, L]$.
\end{theorem}

Following are the direct consequences of Corollary \ref{Without bad cycle for whole vector space} and Theorem \ref{iff condition for Bilinear maps}. 

\begin{corollary}\label{Without bad cycle for whole Lie algebra}
Let $F \ne \mathbb{F}_2$ and $L$ be a Lie algebra over $F$ with finite-dimensional quotient Lie algebra $L / Z(L)$. Let $\mathcal{B}_L$ be a basis of $L / Z(L)$ such that the graph $\Gamma(\mathcal B_L)$ does not contain bad cycles, then $[L, L] = L'$.  
\end{corollary}

\begin{theorem}\label{iff condition for Lie algebra}
Let $F \ne \mathbb{F}_2$ and $L$ be a Lie algebra over $F$ with finite-dimensional quotient Lie algebra $L / Z(L)$. Let $\mathcal{B}_L:= \{\overline{u_1}, \overline{u_2}, \cdots, \overline{u_n}\}$ be a basis of $L / Z(L)$ and let $\Gamma(\mathcal B_L) = (V,E)$ be the associated graph such that the set $\{[u_i,u_j]: e_{i,j} \in E\}$ forms a basis of $L'$. Then $[L, L] = L'$ if and only if $\Gamma(\mathcal B_L)$ does not contain bad cycles.
\end{theorem}

A construction analogous to Remark \ref{Remark for minimality condition for groups is necessary} and $\S$\ref{Remark from graph to group} can be carried out for the case of Lie algebras as well. One may use it to construct infinitely many examples when $[L, L] \neq L'$, a result that is analogous to Corollary \ref{Infinite examples for K(G) not equal to G'}.
\bibliographystyle{amsalpha}
\bibliography{word-maps}

\end{document}